\documentclass[12pt,a4paper,reqno,twoside]{amsart}

\usepackage[english]{babel}
\usepackage{stmaryrd}
\usepackage{dsfont}
\usepackage[symbol*,ragged]{footmisc}
\usepackage[colorlinks,linkcolor=red,anchorcolor=blue,citecolor=blue,urlcolor=blue]{hyperref}
\usepackage{color,xcolor}

\usepackage{geometry}
\usepackage{amssymb}
\usepackage{amsmath}
\usepackage{mathrsfs}
\usepackage{amsfonts}
\usepackage{epsfig}

\usepackage{amsthm}
\usepackage{amsxtra}
\usepackage{bbding}
\usepackage{epsfig}
\usepackage{graphicx}
\usepackage{latexsym}
\usepackage{mathbbol}
\usepackage{bbold}

\usepackage{pifont}
\usepackage{wasysym}
\usepackage{skull}
\usepackage{float}

\usepackage{esint}

\DeclareSymbolFontAlphabet{\mathbb}{AMSb}
\DeclareSymbolFontAlphabet{\mathbbol}{bbold}

\usepackage{amscd}
\usepackage[all]{xy}
\allowdisplaybreaks[4]
\usepackage{setspace}

\theoremstyle{plain}
\newtheorem{theorem}{\normalfont\scshape Theorem}[section]
\newtheorem{proposition}{\normalfont\scshape Proposition}[section]
\newtheorem{lemma}[proposition]{\normalfont\scshape Lemma}

\newtheorem*{corollary*}{\normalfont\scshape Corollary}

\theoremstyle{remark}
\newtheorem*{remark*}{\normalfont\scshape Remark}
\newtheorem*{notation}{\normalfont\scshape Notation}

\numberwithin{equation}{section}
\addtocounter{footnote}{1}

\renewcommand{\footnoterule}{
	\kern -3pt
	\hrule width 2.5in height 0.4pt
	\kern 3pt
}

\makeatletter
\@ifundefined{MakeUppercase}{}{}
\makeatother

\begin{document}
	
\title[On a system of two Diophantine inequalities with six prime variables]
      {On a system of two Diophantine inequalities with six prime variables}
	
\author[Linji Long, Jinjiang Li, Min Zhang, Rui Sun]
{Linji Long \quad \& \quad Jinjiang Li \quad \& \quad Min Zhang \quad \& \quad Rui Sun}
	
\address{[Linji Long] Department of Mathematics, China University of Mining and Technology,
		              Beijing 100083, People's Republic of China}
	
\email{\textcolor{black}{linji.long.math@gmail.com}}

\address{[Jinjiang Li] (Corresponding author) Department of Mathematics, China University of Mining and Technology,
		               Beijing 100083, People's Republic of China}
	
\email{\textcolor{black}{jinjiang.li.math@gmail.com}}

\address{[Min Zhang] School of Applied Science, Beijing Information Science and Technology University,
		             Beijing 100192, People's Republic of China  }
	
\email{\textcolor{black}{min.zhang.math@gmail.com}}

\address{[Rui Sun] Department of Mathematics, China University of Mining and Technology,
		           Beijing 100083, People's Republic of China}
	
\email{\textcolor{black}{rui.sun.math.nt@gmail.com}}
	
\date{}
	
\footnotetext[1]{Jinjiang Li is the corresponding author. 
		
\quad\,\,
{\textbf{MR(2020) Subject Classification}}: 11D75, 11P05, 11L07, 11L20
		
}

\begin{abstract}
Suppose that $c,d,\alpha,\beta$ are real numbers satisfying the inequalities $1<d<c<79/71$ and $1<\alpha<\beta<6^{1-d/c}$. In this paper, it is proved that, for sufficiently large real numbers $N_1$ and $N_2$ subject to $\alpha\leqslant N_2/N_1^{d/c}\leqslant\beta$, the following Diophantine inequalities system
\begin{align*}
\begin{cases}
|p_1^c+p_2^c+p_3^c+p_4^c+p_5^c+p_6^c-N_1|<\varepsilon_1 (N_1) \\
|p_1^d+p_2^d+p_3^d+p_4^d+p_5^d+p_6^d-N_2|<\varepsilon_2 (N_2)
\end{cases}
\end{align*}
is solvable in prime variables $p_1, p_2, p_3, p_4, p_5, p_6$, where
\begin{align*}
\begin{cases}
\varepsilon_1 (N_1)=N_1^{-(1/c)(79/71-c)} (\log N_1)^{201}, \\
\varepsilon_2 (N_2)=N_2^{-(1/d)(79/71-d)} (\log N_2)^{201} .
\end{cases}
\end{align*}
This result constitutes an improvement upon the previous result of Han--Liu--Zhang \cite{Han-Liu-Zhang-2021}.
\end{abstract}
	
\maketitle

\section{Introduction and main result}
The famous Waring--Goldbach problem in additive number theory states that every large integer $N$ satisfying appropriate congruent conditions should be represented as the sum of $s$ $k$--th powers of prime numbers, i.e.,
\begin{equation}\label{Waring-Goldbach-General}
N=p_1^k+p_2^k+\dots+p_s^k.
\end{equation}
In this topic, many mathematicians have derived many splendid results. For instance, in 1937, Vinogradov \cite{Vinogradov-1937} proved that such a representation of the type (\ref{Waring-Goldbach-General}) exists for every sufficiently large odd integer with $k=1,s=3$. Moreover, in 1938, Hua \cite{Hua-1938} showed that (\ref{Waring-Goldbach-General}) is solvable for every sufficiently large integer $N$ satisfying
$N\equiv 5\pmod{24}$ with $k=2,s=5$.

In 1952, Piatetski--Shapiro \cite{Piatetski-Shapiro-1952} studied the following analog of the Waring--Goldbach problem. Suppose that $c>1$ is not an integer and $\varepsilon$ is a small positive number. Denote by $H(c)$ the smallest natural number $s$ such that, for every sufficiently large real number $N$, the Diophantine inequality
\begin{equation}\label{WG-analogue}
\big|p_1^c+p_2^c+\cdots+p_s^c-N\big|<\varepsilon
\end{equation}
is solvable in primes $p_1,p_2,\dots,p_s$. Then it was proved in \cite{Piatetski-Shapiro-1952} that
\begin{equation*}
\limsup_{c\to+\infty}\frac{H(c)}{c\log c}\leqslant 4.
\end{equation*}
Also, in \cite{Piatetski-Shapiro-1952}, Piatetski--Shapiro considered the case $s=5$ in (\ref{WG-analogue}) and proved that $H(c)\leqslant5$ for $1<c<3/2$. Later, the upper bound $3/2$ for $H(c)\leqslant5$ was improved
successively to
\begin{equation*}
\frac{14142}{8923},\quad \frac{1+\sqrt{5}}{2},\quad \frac{81}{40}, \quad 2.041, \quad \frac{378}{181}
\end{equation*}
by Zhai and Cao \cite{Zhai-Cao-2003}, Garaev \cite{Garaev-2003}, Zhai and Cao \cite{Zhai-Cao-2007}, Baker and Weingartner \cite{Baker-Weingartner-2013},  Baker \cite{Baker-2021}, respectively.
	
In 1995, Tolev \cite{Tolev-1995} considered the system of two inequalities with $s=5$ as follows
\begin{equation}
\begin{cases}\label{5-var-system}
\big|p_1^c+p_2^c+p_3^c+p_4^c+p_5^c-N_1\big|<\varepsilon_1(N_1), \\
\big|p_1^d+p_2^d+p_3^d+p_4^d+p_5^d-N_2\big|<\varepsilon_2(N_2),
\end{cases}
\end{equation}
where $c$ and $d$ are different numbers greater than one but close to one and
$\varepsilon_1(N_1),\varepsilon_2(N_2)$ tend to zero as $N_1$ and $N_2$ tend to infinity. Of course, one
has to impose a condition on the orders of $N_1$ and $N_2$ because of the inequality
\begin{equation*}
(x_1^c+\dots+x_5^c)^{d/c}\leqslant x_1^d+\dots+x_5^d\leqslant5^{1-d/c}(x_1^c+\dots+x_5^c)^{d/c}
\end{equation*}
which holds for every positive $x_1,\dots,x_5$ provided $1<d<c$. Tolev \cite{Tolev-1995} proved that if
$c,d,\alpha,\beta$ are real numbers satisfying
\begin{equation*}
1<d<c<35/34,\qquad 1<\alpha<\beta<5^{1-d/c},
\end{equation*}
then there exist numbers $N_1^{(0)}$ and $N_2^{(0)}$, which depend on $c,d,\alpha,\beta$, such that for all real numbers $N_1,N_2$ subject to $N_1>N_1^{(0)},N_2>N_2^{(0)}$ and
\begin{equation*}
\alpha\leqslant N_2/N_1^{d/c}\leqslant\beta,
\end{equation*}
the system (\ref{5-var-system}) has prime solutions $p_1,\dots,p_5$ for
\begin{equation*}
\begin{cases}
\varepsilon_1(N_1)=N_1^{-(1/c)(35/34-c)}(\log N_1)^{12}, \\
\varepsilon_2(N_2)=N_2^{-(1/d)(35/34-d)}(\log N_2)^{12}.
\end{cases}
\end{equation*}
Later, in 2000, Zhai \cite{Zhai-2000} enhanced the result of Tolev \cite{Tolev-1995} and enlarged the range to $1<d<c<25/24$. Recently, Zhang--Li--Long--Yang \cite{Zhang-Li-Long-Yang-2026} improved the result of Zhai \cite{Zhai-2000} and showed that (\ref{5-var-system}) is solvable for $1<d<c<39/37$.
	
In 2021, Han--Liu--Zhang \cite{Han-Liu-Zhang-2021} extended the investigation to the system of two inequalities with $s=6$, i.e.,
\begin{equation*}
\begin{cases}
|p_1^c + p_2^c + p_3^c + p_4^c + p_5^c + p_6^c - N_1| < \varepsilon_1(N_1), \\
|p_1^d + p_2^d + p_3^d + p_4^d + p_5^d + p_6^d - N_2| < \varepsilon_2(N_2).
\end{cases}
\end{equation*}
They proved that, for sufficiently large $N_1$ and $N_2$, the above system has infinitely many prime solutions, provided that
\begin{equation*}
1 < d < c < {128}/{119}, \qquad 1 < \alpha < \beta < 6^{1-d/c},
\end{equation*}
 $\alpha \leqslant N_2/N_1^{d/c} \leqslant \beta$ and
\begin{align*}
\begin{cases}
\varepsilon_1 (N_1)=N_1^{-(1/c)(128/119-c)} (\log N_1)^{201}, \\
\varepsilon_2 (N_2)=N_2^{-(1/d)(128/119-d)} (\log N_2)^{201} .
\end{cases}
\end{align*}
	
In this paper, we shall enhance the result of Han--Liu--Zhang \cite{Han-Liu-Zhang-2021} and establish the following theorem.
\begin{theorem}\label{Theorem-1}
Suppose that $c, d, \alpha, \beta$ are real numbers satisfying the inequalities
\begin{equation}\label{eq1_4}
1<d<c<79/71, \qquad 1<\alpha<\beta<6^{1-d/c} .
\end{equation}
Then there exist numbers $N_1^{(0)}$ and $N_2^{(0)}$, which depend on $c, d, \alpha, \beta$, such that for all real numbers $N_1, N_2$ subject to $N_1>N_1^{(0)}$, $N_2>N_2^{(0)}$ and
\begin{equation}\label{eq1_5}
\alpha \leqslant N_2 / N_1^{d/c} \leqslant \beta,
\end{equation}
the system
\begin{align*}
\begin{cases}
|p_1^c+p_2^c+p_3^c+p_4^c+p_5^c+p_6^c-N_1|<\varepsilon_1(N_1) \\
|p_1^d+p_2^d+p_3^d+p_4^d+p_5^d+p_6^d-N_2|<\varepsilon_2(N_2)
\end{cases}
\end{align*}
with
\begin{align*}
\begin{cases}
\varepsilon_1(N_1)=N_1^{-(1/c)(79/71-c)}(\log N_1)^{201} \\
\varepsilon_2(N_2)=N_2^{-(1/d)(79/71-d)}(\log N_2)^{201}
\end{cases}
\end{align*}
is solvable in prime variables $p_1, p_2, p_3, p_4, p_5, p_6$.
\end{theorem}
	
\begin{remark*}
In order to compare our result with the previous result of Han--Liu--Zhang \cite{Han-Liu-Zhang-2021}, we list the numerical results as follows:
\begin{equation*}
\frac{128}{119}=1.0756302 \dots ;  \qquad \frac{79}{71}=1.1267605\dots .
\end{equation*}
\end{remark*}
	
\begin{notation}
Let $c,d,\alpha,\beta$ be numbers satisfying (\ref{eq1_4}). The letter $p$, with or without subscript, always denotes a prime number. As usual, we use $\mu(n)$ and $\Lambda(n)$ to denote M\"{o}bius' function and von Mangoldt's function, respectively. Throughout this paper, the constants in $O$--terms and $\ll$--symbols are absolute or at most depend on $c,d,\alpha,\beta$. $f(x)\ll g(x)$ means that $f(x)= O(g(x))$; $f(x) \asymp g(x)$ means that $f(x)\ll g(x)\ll f(x)$. $N_1$ and $N_2$ are sufficiently large real numbers subject to (\ref{eq1_5}). Set
\begin{gather*}
X=N_1^{1 / c}, \qquad \varepsilon_1=X^{-(79/71-c)}(\log X)^{401},
\qquad \varepsilon_2=X^{-(79/71-d)}(\log X)^{401},
					\nonumber \\
K_1=\varepsilon_1^{-1} \log X,
\qquad K_2=\varepsilon_2^{-1} \log X,
\qquad \tau_1=X^{3/4-c-\eta},
\qquad \tau_2=X^{3/4-d-\eta},
\end{gather*}
where $\eta$ is a positive number which is sufficiently small in terms of $c$ and $d$. $e(t)=\exp(2\pi it)$, $\phi(t)=e^{-\pi t^2}$, $\phi_\delta(t)=\delta\cdot\phi(\delta t)$. $\mathds{1}_{[-1,1]}(t)$ is the characteristic function of the interval $[-1,1]$. Denote by $\lambda$ a sufficiently small positive number, depending on $c, d, \alpha,\beta$, whose value will be determined more precisely by Lemma 1 of Tolev \cite{Tolev-1995}. Define
\begin{gather}
\mathscr{B}=\sum_{\lambda X<p_1, \dots, p_6 \leqslant X}\Bigg(\prod_{i=1}^6 \log p_i\Bigg)
\mathds{1}_{[-1,1]}\bigg(\frac{p_1^c+\cdots+p_6^c-N_1}{\varepsilon_1 \log X}\bigg)
\mathds{1}_{[-1,1]}\bigg(\frac{p_1^d+\cdots+p_6^d-N_2}{\varepsilon_2 \log X}\bigg)
					\nonumber\\
S(x, y)=\sum_{\lambda X<p \leqslant X}(\log p) e(p^c x+p^d y),\qquad
\mathcal{T}_\alpha(x)=\sum_{\lambda X<n \leqslant X} e(n^\alpha x)\label{eq1_6}
                    \nonumber\\
\mathscr{D}=\int_{-\infty}^{+\infty}\int_{-\infty}^{+\infty}S^6(x, y)e(-N_1x-N_2y)\phi_{\varepsilon_1}(x) \phi_{\varepsilon_2}(y)\mathrm{d}x\mathrm{d}y.
\end{gather}
\end{notation}

	\section{Proof of Theorem \ref{Theorem-1}}
	
	The Theorem follows if one shows that $\mathscr{B}$ tends to infinity as $X$ tends to infinity. We first give a lemma as follows.
	\begin{lemma}\label{Tolev-1995}
	The function $\phi(t)=e^{-\pi t^2}$ has the following properties
		\begin{align*}
			\mathrm{(i)} \quad & \phi(x) = \int_{-\infty}^{+\infty} \phi(t)e(-xt)\mathrm{d}t; \\
			\mathrm{(ii)} \quad & \mathds{1}_{[-1,1]}(t/\varrho) \geqslant \phi(t) - e^{-\pi \varrho^2}, \quad \text{for any } \varrho > 0; \\
			\mathrm{(iii)} \quad & \phi(t) \geqslant e^{-\pi}, \quad \text{for } |t| \leqslant 1.
		\end{align*}
	\end{lemma}
	\begin{proof}
		See Lemma 2 of Tolev \cite{Tolev-1995}.
	\end{proof}
	Trivially, by Lemma 2.1, we get
	\begin{align}\label{eq2_1}
		\mathscr{B}
		\geqslant & \,\, \sum_{\lambda X<p_1, \dots, p_6 \leqslant X}
				\Bigg(\prod_{i=1}^6 \log p_i\Bigg)
				\Bigg(\phi\bigg(\frac{p_1^c+\cdots+p_6^c-N_1}{\varepsilon_1}\bigg)
					-e^{-\pi(\log X)^2}\Bigg)
						\nonumber\\
		& \,\, \hspace{5em} \times
				\Bigg(\phi\bigg(\frac{p_1^d+\cdots+p_6^d-N_2}{\varepsilon_2}\bigg)
					-e^{-\pi(\log X)^2}\Bigg)
						\nonumber\\
		= & \,\, \sum_{\lambda X<p_1, \dots, p_6 \leqslant X}
				\Bigg(\prod_{i=1}^6 \log p_i\Bigg)
				\phi\bigg(\frac{p_1^c+\cdots+p_6^c-N_1}{\varepsilon_1}\bigg)
						\nonumber\\
		& \,\, \hspace{5em} \times
				\phi\bigg(\frac{p_1^d+\cdots+p_6^d-N_2}{\varepsilon_2}\bigg)
				+ O\Big(X^6 e^{-\pi(\log X)^2}\Big)
						\nonumber\\
		= & \,\, \sum_{\lambda X<p_1, \dots, p_6 \leqslant X}
				\Bigg(\prod_{i=1}^6 \log p_i\Bigg)
				\int_{-\infty}^{+\infty} \phi(t_1)
				e\bigg(\frac{(p_1^c+\cdots+p_6^c-N_1) t_1}{\varepsilon_1}\bigg) \mathrm{d} t_1
						\nonumber\\
		& \,\, \hspace{5em} \times
				\int_{-\infty}^{+\infty} \phi(t_2)
				e\bigg(\frac{(p_1^d+\cdots+p_6^d-N_2) t_2}{\varepsilon_2}\bigg) \mathrm{d} t_2
				+ O(1)
						\nonumber\\
		= & \,\, \sum_{\lambda X<p_1, \dots, p_6 \leqslant X}
				\Bigg(\prod_{i=1}^6 \log p_i\Bigg)
				\int_{-\infty}^{+\infty} e\big((p_1^c+\cdots+p_6^c-N_1) x\big) \varepsilon_1 \phi(\varepsilon_1 x) \mathrm{d} x
						\nonumber\\
		& \,\, \hspace{5em} \times
				\int_{-\infty}^{+\infty} e\big((p_1^d+\cdots+p_6^d-N_2) y\big)
				\varepsilon_2 \phi(\varepsilon_2 y) \mathrm{d} y
				+ O(1)
						\nonumber\\
		= & \,\, \int_{-\infty}^{+\infty} \int_{-\infty}^{+\infty}
				\Bigg(\sum_{\lambda X<p \leqslant X}(\log p) e(p^c x+p^d y)\Bigg)^6
				e(-N_1 x-N_2 y) \phi_{\varepsilon_1}(x)
				\phi_{\varepsilon_2}(y) \mathrm{d} x \mathrm{d} y
				+ O(1)
						\nonumber\\
		= & \,\, \mathscr{D} + O(1) .
	\end{align}
	Now, we divide the plane into three regions: $\Omega_1$---a neighbourhood of the origin, $\Omega_2$---an intermediate region, and $\Omega_3$---a trivial region, as follows
	\begin{align*}
		& \Omega_1=\big\{(x, y): \max \big(|x| \tau_1^{-1},|y| \tau_2^{-1}\big)<1\big\}, \\
		& \Omega_2=\big\{(x, y): \max \big(|x| \tau_1^{-1},|y| \tau_2^{-1}\big) \geqslant 1,
				\, \max \big(|x| K_1^{-1},|y| K_2^{-1}\big) \leqslant 1\big\}, \\
		& \Omega_3=\big\{(x, y): \max \big(|x| K_1^{-1},|y| K_2^{-1}\big)>1\big\} .
	\end{align*}	
	Correspondingly, we represent the integral $\mathscr{D}$ as
	\begin{equation}\label{eq2_2}
		\mathscr{D}=\mathscr{D}_1+\mathscr{D}_2+\mathscr{D}_3,
	\end{equation}
	where $\mathscr{D}_i$ denotes the contribution to the integral $\mathscr{D}$ in (\ref{eq2_2}) arising from the set $\Omega_i$. The result of (\ref{eq2_1}) implies that it suffices to prove that $\mathscr{D}$ tends to infinity as $X$ tends to infinity. The last statement is a consequence of (\ref{eq2_2}) and of the following three inequalities
	\begin{gather}\label{eq2_3}
		\mathscr{D}_1 \gg \varepsilon_1 \varepsilon_2 X^{6-c-d}, \\
		\mathscr{D}_2 \ll \varepsilon_1 \varepsilon_2 X^{6-c-d}(\log X)^{-1}, \nonumber\\
		\mathscr{D}_3 \ll 1 . \nonumber
	\end{gather}
	According to Lemma 14 of Han--Liu--Zhang \cite{Han-Liu-Zhang-2021}, we obtain (\ref{eq2_3}). For the upper bound estimate of $\mathscr{D}_3$, it follows from the definition of $\Omega_3$ that
	\begin{align*}
		\mathscr{D}_3
		\ll & \,\, \int_{-\infty}^{+\infty} \phi_{\varepsilon_1}(x) \mathrm{d} x
				\int_{|y| \geqslant K_2}|S(x, y)|^6 \phi_{\varepsilon_2}(y) \mathrm{d} y
				+ \int_{|y|<K_2} \phi_{\varepsilon_2}(y) \mathrm{d} y
				\int_{|x| \geqslant K_1}|S(x, y)|^6 \phi_{\varepsilon_1}(x) \mathrm{d} x
							\\
		\ll & \,\, X^6 \Bigg(\int_{|y| \geqslant K_2} \varepsilon_2
					e^{-\pi \varepsilon_2^2 y^2} \mathrm{d} y
				+ \int_{|x| \geqslant K_1} \varepsilon_1
					e^{-\pi \varepsilon_1^2 x^2} \mathrm{d} x\Bigg)
							\\
		\ll & \,\, X^6(\log X)^{-1} e^{-\pi(\log X)^2} \ll 1,
	\end{align*}
	where we use the trivial estimate
	\begin{equation*}
		\int_{-\infty}^{+\infty} \phi_{\varepsilon_1}(x) \mathrm{d} x=\int_{-\infty}^{+\infty} \varepsilon_1 \phi(\varepsilon_1 x) \mathrm{d} x=\int_{-\infty}^{+\infty} \phi(t) \mathrm{d} t=\int_{-\infty}^{+\infty} e^{-\pi t^2} \mathrm{~d} t \ll 1,
	\end{equation*}
	and
	\begin{equation*}
		\int_{|y|<K_2} \phi_{\varepsilon_2}(y) \mathrm{d} y \ll \int_{-\infty}^{+\infty} \phi_{\varepsilon_2}(y) \mathrm{d} y \ll 1 .
	\end{equation*}
	In the rest of this section, we focus on the upper bound estimate of $\mathscr{D}_2$.
	
	\begin{lemma}\label{S(x,y)-int-mean-1}
		For $S(x,y)$ defined as in (\ref{eq1_6}), there holds
		\begin{align*}
			\emph{(i)}  & \,\, \int_{-\tau_1}^{\tau_1}|S(x,y)|^2\mathrm{d}x\ll X^{2-c}(\log X)^3, \quad
			\textrm{uniformly for all $y\in\mathbb{R}$};
			\nonumber \\
			\emph{(ii)} & \,\, \int_{-\tau_2}^{\tau_2}|S(x,y)|^2\mathrm{d}y\ll X^{2-d}(\log X)^3, \quad
			\textrm{uniformly for all $x\in\mathbb{R}$};
			\nonumber \\
			\emph{(iii)} & \,\, \int_{n}^{n+1}|S(x,y)|^2\mathrm{d}x\ll X(\log X)^3, \quad
			\textrm{uniformly for all $y\in\mathbb{R}$ and $n\in\mathbb{N}$};
			\nonumber \\
			\emph{(iv)} & \,\, \int_{n}^{n+1}|S(x,y)|^2\mathrm{d}y\ll X(\log X)^3, \quad
			\textrm{uniformly for all $x\in\mathbb{R}$ and $n\in\mathbb{N}$}.
		\end{align*}
	\end{lemma}
	\begin{proof}
		We follow the process of the arguments of Lemma 7 of Tolev \cite{Tolev-1995}. We only give the details of the proof
		of (i). The estimates (ii), (iii) and (iv) can be established likewise. We have
		\begin{align}\label{eq2_4}
			\int_{-\tau_1}^{\tau_1}|S(x,y)|^2\mathrm{d}x
			= & \,\, \sum_{\lambda X<p_1,p_2\leqslant X}(\log p_1)(\log p_2)e\big((p_1^d-p_2^d)y\big)
			\int_{-\tau_1}^{\tau_1}e\big((p_1^c-p_2^c)x\big)\mathrm{d}x
			\nonumber \\
			\ll & \,\, \sum_{\lambda X<p_1,p_2\leqslant X}(\log p_1)(\log p_2)
			\bigg|\int_{-\tau_1}^{\tau_1}e\big((p_1^c-p_2^c)x\big)\mathrm{d}x\bigg|
			\nonumber \\
			\ll & \,\, \sum_{\lambda X<p_1,p_2\leqslant X}(\log p_1)(\log p_2)\min\bigg(\tau_1,\frac{1}{|p_1^c-p_2^c|}\bigg)
			\nonumber \\
			\ll & \,\, \mathcal{U}\tau_1(\log X)^2+\mathcal{V}(\log X)^2,
		\end{align}
		where
		\begin{equation*}
			\mathcal{U}=\sum_{\substack{\lambda X<n_1,n_2\leqslant X\\ |n_1^c-n_2^c|\leqslant1/\tau_1}}1,\qquad\qquad
			\mathcal{V}=\sum_{\substack{\lambda X<n_1,n_2\leqslant X\\ |n_1^c-n_2^c|>1/\tau_1}}\frac{1}{|n_1^c-n_2^c|}.
		\end{equation*}
		Obviously, one has
		\begin{equation*}
			\mathcal{U}\ll\mathop{\sum_{\lambda X<n_1\leqslant X}
				\sum_{\lambda X<n_2\leqslant X}}_{(n_1^c-1/\tau_1)^{1/c}\leqslant n_2\leqslant(n_1^c+1/\tau_1)^{1/c}}1
			\ll\sum_{\lambda X<n_1\leqslant X}\big(1+(n_1^c+1/\tau_1)^{1/c}-(n_1^c-1/\tau_1)^{1/c}\big),
		\end{equation*}
		and by the mean--value theorem
		\begin{equation}\label{eq2_5}
			\mathcal{U}\ll X+\frac{1}{\tau_1}X^{2-c}.
		\end{equation}
		Trivially, by a splitting argument, there holds $\mathcal{V}\ll\sum_{\ell}\mathcal{V}_\ell$, where
		\begin{equation}\label{eq2_6}
			\mathcal{V}_\ell=\sum_{\substack{\lambda X<n_1,n_2\leqslant X\\ \ell<|n_1^c-n_2^c|\leqslant2\ell}}\frac{1}{|n_1^c-n_2^c|},
		\end{equation}
		and $\ell$ takes the values $2^k/\tau_1,\,k=0,1,2,\dots$, with $\ell\ll X^c$. Therefore, we get
		\begin{equation*}
			\mathcal{V}_\ell\ll\frac{1}{\ell}
			\mathop{\sum_{\lambda X<n_1\leqslant X}
				\sum_{\lambda X<n_2\leqslant X}}_{(n_1^c+\ell)^{1/c}\leqslant n_2\leqslant(n_1^c+2\ell)^{1/c}}1.
		\end{equation*}
		For $\ell\geqslant1/\tau_1$ and $\lambda X<n_1\leqslant X$, it is easy to see that
		\begin{equation*}
			(n_1^c+2\ell)^{1/c}-(n_1^c+\ell)^{1/c}>1.
		\end{equation*}
		Hence
		\begin{equation}\label{eq2_7}
			\mathcal{V}_\ell\ll\frac{1}{\ell}\sum_{\lambda X<n_1\leqslant X}
			\big((n_1^c+2\ell)^{1/c}-(n_1^c+\ell)^{1/c}\big)\ll X^{2-c}
		\end{equation}
		by the mean--value theorem. According to (\ref{eq2_4})--(\ref{eq2_7}), the conclusion follows.
	\end{proof}

	\begin{lemma}\label{S(x,y)-int-mean-2}
		For $S(x,y)$ defined as in (\ref{eq1_6}), there holds
		\begin{align*}
			\emph{(i)}  & \,\, \int_{|x|\leqslant K_1}|S(x,y)|^2\phi_{\varepsilon_1}(x)\mathrm{d}x\ll X(\log X)^4;
			\nonumber \\
			\emph{(ii)} & \,\, \int_{|y|\leqslant K_2}|S(x,y)|^2\phi_{\varepsilon_2}(y)\mathrm{d}y\ll X(\log X)^4.
		\end{align*}
	\end{lemma}
	\begin{proof}
		For (i), it follows from (i) and (iii) of Lemma \ref{S(x,y)-int-mean-1} and the trivial estimate $\phi_{\varepsilon_1}(x)\ll\varepsilon_1$ that
		\begin{align*}
			\int_{|x|\leqslant K_1}|S(x,y)|^2\phi_{\varepsilon_1}(x)\mathrm{d}x
			= & \,\, \int_{|x|\leqslant \tau_1}|S(x,y)|^2\phi_{\varepsilon_1}(x)\mathrm{d}x
			+\int_{\tau_1<|x|\leqslant K_1}|S(x,y)|^2\phi_{\varepsilon_1}(x)\mathrm{d}x
			\nonumber \\
			\ll & \,\, \varepsilon_1X^{2-c}(\log X)^3+\varepsilon_1\sum_{0\leqslant n\leqslant K_1}
			\int_n^{n+1}|S(x,y)|^2\mathrm{d}x
			\nonumber \\
			\ll & \,\, \varepsilon_1X^{2-c}(\log X)^3+\varepsilon_1K_1X(\log X)^3  \ll X(\log X)^4.
		\end{align*}
		Similarly, for (ii), it follows from (ii) and (iv) of Lemma \ref{S(x,y)-int-mean-1} and the trivial estimate $\phi_{\varepsilon_2}(y)\ll\varepsilon_2$ that
		\begin{align*}
			\int_{|y|\leqslant K_2}|S(x,y)|^2\phi_{\varepsilon_2}(y)\mathrm{d}y
			= & \,\, \int_{|y|\leqslant \tau_2}|S(x,y)|^2\phi_{\varepsilon_2}(y)\mathrm{d}y
			+\int_{\tau_2<|y|\leqslant K_2}|S(x,y)|^2\phi_{\varepsilon_2}(y)\mathrm{d}y
			\nonumber \\
			\ll & \,\, \varepsilon_2X^{2-d}(\log X)^3+\varepsilon_2\sum_{0\leqslant n\leqslant K_2}
			\int_n^{n+1}|S(x,y)|^2\mathrm{d}y
			\nonumber \\
			\ll & \,\, \varepsilon_2X^{2-d}(\log X)^3+\varepsilon_2K_2X(\log X)^3  \ll X(\log X)^4.
		\end{align*}
		This completes the proof of Lemma \ref{S(x,y)-int-mean-2}.
	\end{proof}
	
	\begin{lemma}\label{Graham-Kolesnik-exponential}
		Suppose that $f(x):[a, b]\to \mathbb{R}$ has continuous derivatives of arbitrary order on $[a,b]$,
		where $1\leqslant a <b\leqslant2a$. Suppose further that
		\begin{equation*}
			\big|f^{(j)}(x)\big|\asymp\lambda_1a^{1-j},\quad j\geqslant1,\quad x\in[a,b].
		\end{equation*}
		Then, for any exponential pair $(\kappa,\lambda)$, we have
		\begin{equation*}
			\sum_{a<n\leqslant b}e(f(n))\ll\lambda_1^\kappa a^\lambda+\lambda_1^{-1}.
		\end{equation*}
	\end{lemma}
	\begin{proof}
		See (3.3.4) of Graham and Kolesnik \cite{Graham-Kolesnik-book}.
	\end{proof}
	
	\begin{lemma}\label{expo-full-inte}
		For $S(x,y)$ defined as in (\ref{eq1_6}), there holds
		\begin{equation*}
			\int_{-\infty}^{+\infty}\int_{-\infty}^{+\infty}\big|S(x,y)\big|^4\phi_{\varepsilon_1}(x)
			\phi_{\varepsilon_2}(y)\mathrm{d}x\mathrm{d}y\ll X^2(\log X)^6.
		\end{equation*}
	\end{lemma}
	\begin{proof}
		See Lemma 14 of Tolev \cite{Tolev-1995}.
	\end{proof}
	
	\begin{lemma}\label{S(x,y)-Omega_2}
		For $1<d<c<91/71$ and $(x, y) \in \Omega_2$, we have
		\begin{equation*}
			S(x, y) \ll X^{66/71}(\log X)^{205} .
		\end{equation*}
	\end{lemma}
	\begin{proof}
		The proof of this lemma will be given in Section \ref{expo-esti-sec}.
	\end{proof}
	
	Now, we use the iterative argument to give the upper bound estimate of $\mathscr{D}_2$. By the definition of $\Omega_2$, there holds
	\begin{align}\label{eq2_8}
		\mathscr{D}_2
		= & \,\, \iint\limits_{\substack{|x| \leqslant K_1 \\
						\tau_2 \leqslant|y| \leqslant K_2}}
				S^6(x, y) e(-N_1 x-N_2 y)
				\phi_{\varepsilon_1}(x) \phi_{\varepsilon_2}(y)
				\mathrm{d} x \mathrm{d} y
						\nonumber\\
		 & \,\, + \iint\limits_{\substack{|y| \leqslant \tau_2 \\
						\tau_1 \leqslant|x| \leqslant K_1}}
				S^6(x, y) e(-N_1 x-N_2 y)
				\phi_{\varepsilon_1}(x) \phi_{\varepsilon_2}(y)
				\mathrm{d} x \mathrm{d} y
						\nonumber\\
		\ll & \,\, \iint\limits_{\substack{|x| \leqslant K_1 \\
					\tau_2 \leqslant|y| \leqslant K_2}}
				|S(x, y)|^6 \phi_{\varepsilon_1}(x) \phi_{\varepsilon_2}(y)
				\mathrm{d} x \mathrm{d} y
			+ \iint\limits_{\substack{|y| \leqslant K_2 \\
					\tau_1 \leqslant|x| \leqslant K_1}}
				|S(x, y)|^6 \phi_{\varepsilon_1}(x) \phi_{\varepsilon_2}(y)
				\mathrm{d} x \mathrm{d} y
						\nonumber\\
		=: & \,\, \mathscr{D}_2^{(1)}+\mathscr{D}_2^{(2)} .				
	\end{align}
	According to the definition of $\mathscr{D}_2^{(1)}$, we obtain
	\begin{equation}\label{eq2_9}
		\mathscr{D}_2^{(1)}=\int_{|x| \leqslant K_1} \phi_{\varepsilon_1}(x) \mathrm{d} x \int_{\tau_2 \leqslant|y| \leqslant K_2}|S(x, y)|^6 \phi_{\varepsilon_2}(y) \mathrm{d} y.
	\end{equation}
	For the innermost integral on the right--hand side of (\ref{eq2_9}), it follows from the definition of $S(x, y)$ that
	\begin{align*}
		& \,\, \bigg|\int_{\tau_2 \leqslant|y| \leqslant K_2}
				|S(x, y)|^6 \phi_{\varepsilon_2}(y) \mathrm{d} y \bigg|
					\\
		= & \,\, \bigg|\int_{\tau_2 \leqslant|y| \leqslant K_2}
				S(x, y) \overline{S(x, y)} |S(x, y)|^4
				\phi_{\varepsilon_2}(y) \mathrm{d} y \bigg|
					\\
		= & \,\, \bigg|\sum_{\lambda X<p \leqslant X}(\log p) e(p^c x)
				\int_{\tau_2 \leqslant|y| \leqslant K_2}
				\overline{S(x, y)} |S(x, y)|^4
				e(p^d y) \phi_{\varepsilon_2}(y) \mathrm{d} y \bigg|
					\\
		\leqslant & \,\, (\log X) \sum_{\lambda X<p \leqslant X}
				\bigg|\int_{\tau_2 \leqslant|y| \leqslant K_2}
				\overline{S(x, y)} |S(x, y)|^4
				e(p^d y) \phi_{\varepsilon_2}(y) \mathrm{d} y \bigg|
					\\
		\leqslant & \,\, (\log X) \sum_{\lambda X<n \leqslant X}
				\bigg|\int_{\tau_2 \leqslant|y| \leqslant K_2}
				\overline{S(x, y)} |S(x, y)|^4
				e(n^d y) \phi_{\varepsilon_2}(y) \mathrm{d} y \bigg|,
	\end{align*}
	which combined with Cauchy's inequality yields that
	\begin{align}\label{eq2_10}
			& \,\, \bigg|\int_{\tau_2 \leqslant|y| \leqslant K_2} |S(x, y)|^6 \phi_{\varepsilon_2}(y) \mathrm{d} y\bigg|^2
					\nonumber\\
		\ll & \,\, X(\log X)^2 \sum_{\lambda X<n \leqslant X}
			\bigg|\int_{\tau_2 \leqslant|y| \leqslant K_2}
			\overline{S(x, y)}|S(x, y)|^4 e(n^d y)
			\phi_{\varepsilon_2}(y) \mathrm{d} y\bigg|^2
					\nonumber\\
		= & \,\,  X(\log X)^2 \sum_{\lambda X<n \leqslant X}
			\int_{\tau_2 \leqslant|y_1| \leqslant K_2}
			S(x, y_1)|S(x, y_1)|^4 e(-n^d y_1)
			\phi_{\varepsilon_2}(y_1) \mathrm{d} y_1
					\nonumber\\
		  & \,\, \times \int_{\tau_2 \leqslant|y_2| \leqslant K_2}
		 	\overline{S(x, y_2)}|S(x, y_2)|^4 e(n^d y_2)
		  	\phi_{\varepsilon_2}(y_2) \mathrm{d} y_2
		  			\nonumber\\
		= & \,\, X(\log X)^2
			\int_{\tau_2 \leqslant|y_1| \leqslant K_2}
			S(x, y_1)|S(x, y_1)|^4
					\nonumber\\
		  & \,\, \times \Bigg(\int_{\tau_2 \leqslant|y_2| \leqslant K_2}
		  	\overline{S(x, y_2)}|S(x, y_2)|^4
		  	\bigg(\sum_{\lambda X<n \leqslant X} e\big(n^d(y_2-y_1)\big)\bigg) \phi_{\varepsilon_2}(y_2) \mathrm{d} y_2 \Bigg)
		  	\phi_{\varepsilon_2}(y_1) \mathrm{d} y_1
		  			\nonumber\\
		\ll & \,\, X(\log X)^2
			\int_{\tau_2 \leqslant|y_1| \leqslant K_2}
			|S(x, y_1)|^5
			\bigg(\int_{\tau_2 \leqslant|y_2| \leqslant K_2}
			|S(x, y_2)|^5 |\mathcal{T}_d(y_2-y_1)|
			\phi_{\varepsilon_2}(y_2) \mathrm{d} y_2\bigg)
			\phi_{\varepsilon_2}(y_1) \mathrm{d} y_1
	\end{align}
	For the innermost integral in (\ref{eq2_10}), we get
	\begin{align}\label{eq2_11}
			& \,\, \int_{\tau_2 \leqslant|y_2| \leqslant K_2}
			|S(x, y_2)|^5|\mathcal{T}_d(y_2-y_1)|
			\phi_{\varepsilon_2}(y_2) \mathrm{d} y_2
					\nonumber\\
		\ll & \,\, \int_{\substack{\tau_2 \leqslant|y_2| \leqslant K_2 \\
				|y_2-y_1| \leqslant X^{-d}}}
			|S(x, y_2)|^5|\mathcal{T}_d(y_2-y_1)|
			\phi_{\varepsilon_2}(y_2) \mathrm{d} y_2
					\nonumber\\
		& \,\, + \int_{\substack{\tau_2 \leqslant|y_2| \leqslant K_2 \\
				X^{-d}<|y_2-y_1| \leqslant 2 K_2}}
			|S(x, y_2)|^5|\mathcal{T}_d(y_2-y_1)|
			\phi_{\varepsilon_2}(y_2) \mathrm{d} y_2.
	\end{align}
	For $|y_2-y_1| \leqslant X^{-d}$, according to trivial estimate $\mathcal{T}_d(y_2-y_1) \ll X$, we get
	\begin{align}\label{eq2_12}
			 \,\, \int_{\substack{\tau_2 \leqslant|y_2| \leqslant K_2 \\
				|y_2-y_1| \leqslant X^{-d}}}
			|S(x, y_2)|^5 |\mathcal{T}_d(y_2-y_1)|
			\phi_{\varepsilon_2}(y_2) \mathrm{d} y_2
		\ll  X  \int_{\substack{\tau_2 \leqslant|y_2| \leqslant K_2 \\
				|y_2-y_1| \leqslant X^{-d}}}
			|S(x, y_2)|^5 \phi_{\varepsilon_2}(y_2) \mathrm{d} y_2.
	\end{align}	
	For $X^{-d}<|y_2-y_1| \leqslant 2 K_2$, it follows from Lemma \ref{Graham-Kolesnik-exponential} a with $(\kappa, \ell)=B A^2 B(0,1)=(2/7, 4/7)$ that
	\begin{equation*}
		\mathcal{T}_d(y_2-y_1) \ll|y_2-y_1|^{2 / 7} X^{2(d+1) / 7}+\frac{1}{|y_2-y_1| X^{d-1}},
	\end{equation*}
	and thus
	\begin{align}\label{eq2_13}
			& \,\, \int_{\substack{\tau_2 \leqslant|y_2| \leqslant K_2 \\
				X^{-d}<|y_2-y_1| \leqslant 2 K_2}}
			|S(x, y_2)|^5 |\mathcal{T}_d(y_2-y_1)|
			\phi_{\varepsilon_2}(y_2) \mathrm{d} y_2
						\nonumber\\
		\ll & \,\, \int_{\substack{\tau_2 \leqslant|y_2| \leqslant K_2 \\
				X^{-d}<|y_2-y_1| \leqslant 2 K_2}}
				|S(x, y_2)|^5 \bigg(|y_2-y_1|^{2 / 7} X^{2(d+1) / 7}
			+ \frac{1}{|y_2-y_1| X^{d-1}}\bigg)
				\phi_{\varepsilon_2}(y_2) \mathrm{d} y_2
						\nonumber\\
		\ll & \,\, X^{300/497}(\log X)^{-114}
			\int_{\substack{\tau_2 \leqslant|y_2| \leqslant K_2 \\
				X^{-d}<|y_2-y_1| \leqslant 2 K_2}}
				|S(x, y_2)|^5 \phi_{\varepsilon_2}(y_2) \mathrm{d} y_2
						\nonumber\\
		& \,\, + X^{1-d} \int_{\substack{\tau_2 \leqslant|y_2| \leqslant K_2 \\
				X^{-d}<|y_2-y_1| \leqslant 2 K_2}}
				|S(x, y_2)|^5|y_2-y_1|^{-1}
				\phi_{\varepsilon_2}(y_2) \mathrm{d} y_2 .
	\end{align}
	Therefore, it follows from (\ref{eq2_10})--(\ref{eq2_13}) and Cauchy's inequality that
	\begin{align*}
			& \,\, \int_{\tau_2 \leqslant|y| \leqslant K_2}|S(x, y)|^6 \phi_{\varepsilon_2}(y) \mathrm{d} y \\
		\ll & \,\,  X^{1/2}(\log X)
			\Bigg(\int_{\tau_2 \leqslant |y_1| \leqslant K_2}
			|S(x, y_1)|^5
			\bigg(X \int_{\substack{\tau_2 \leqslant|y_2| \leqslant K_2 \\
				|y_2-y_1| \leqslant X^{-d}}}
			|S(x, y_2)|^5 \phi_{\varepsilon_2}(y_2) \mathrm{d} y_2
						\\
		 & \,\, \hspace{5.6em} + X^{300/497}(\log X)^{-114}
		 	\int_{\substack{\tau_2 \leqslant|y_2| \leqslant K_2 \\
				X^{-d}<|y_2-y_1| \leqslant 2 K_2}}
			|S(x, y_2)|^5 \phi_{\varepsilon_2}(y_2) \mathrm{d} y_2
						\\
		 & \,\, \hspace{5.6em} + X^{1-d}
		 	\int_{\substack{\tau_2 \leqslant|y_2| \leqslant K_2 \\
		 		X^{-d}<|y_2-y_1| \leqslant 2 K_2}}
		 	|S(x, y_2)|^5 |y_2-y_1|^{-1} \phi_{\varepsilon_2}(y_2) \mathrm{d} y_2\bigg)
		 	\phi_{\varepsilon_2} (y_1) \mathrm{d} y_1\Bigg)^{1/2}
		 				\\
		\ll & \,\, X(\log X)
			\bigg(\sup_{\tau_2 \leqslant|y_1| \leqslant K_2}
			\int_{\substack{\tau_2 \leqslant|y_2| \leqslant K_2 \\
					|y_2-y_1| \leqslant X^{-d}}}
			|S(x, y_2)|^5 \phi_{\varepsilon_2}(y_2) \mathrm{d} y_2\bigg)^{1/2}
						\\
		& \,\, \hspace{5em} \times
			\bigg(\int_{\tau_2 \leqslant|y_1| \leqslant K_2}
			|S(x, y_1)|^5 \phi_{\varepsilon_2}(y_1) \mathrm{d} y_1\bigg)^{1/2}
						\\
		& \,\, + X^{797/994}(\log X)^{-56}
			\bigg(\sup_{\tau_2 \leqslant|y_1| \leqslant K_2} \int_{\substack{\tau_2 \leqslant|y_2| \leqslant K_2 \\
					X^{-d}<|y_2-y_1| \leqslant 2 K_2}}
			|S(x, y_2)|^5 \phi_{\varepsilon_2}(y_2) \mathrm{d} y_2\bigg)^{1/2}
						\\
		& \,\, \hspace{5em} \times
			\bigg(\int_{\tau_2 \leqslant|y_1| \leqslant K_2}
			|S(x, y_1)|^5 \phi_{\varepsilon_2}(y_1) \mathrm{d} y_1\bigg)^{1/2}
						\\
		& \,\, + X^{1-d / 2} (\log X)
			\bigg(\sup_{\tau_2 \leqslant|y_1| \leqslant K_2} \int_{\substack{\tau_2 \leqslant|y_2| \leqslant K_2 \\
					X^{-d}<|y_2-y_1| \leqslant 2 K_2}}
			|S(x, y_2)|^5 |y_2-y_1|^{-1}
			\phi_{\varepsilon_2}(y_2) \mathrm{d} y_2\bigg)^{1/2}
						\\
		& \,\, \hspace{5em} \times
			\bigg(\int_{\tau_2 \leqslant|y_1| \leqslant K_2}
			|S(x, y_1)|^5 \phi_{\varepsilon_2}(y_1) \mathrm{d} y_1\bigg)^{1/2}
	\end{align*}
	which combined with (\ref{eq2_9}) yields that
	\begin{align}\label{eq2_14}
		\mathscr{D}_2^{(1)}
		\ll & \,\, X(\log X)
			\int_{|x| \leqslant K_1}
			\bigg(\sup_{\tau_2 \leqslant|y_1| \leqslant K_2}
			\int_{\substack{\tau_2 \leqslant|y_2| \leqslant K_2 \\
				|y_2-y_1| \leqslant X^{-d}}}
			|S(x, y_2)|^5 \phi_{\varepsilon_2}(y_2) \mathrm{d} y_2\bigg)^{1/2}
						\nonumber\\
		& \,\, \hspace{5em} \times
			\bigg(\int_{\tau_2 \leqslant|y_1| \leqslant K_2}
			|S(x, y_1)|^5 \phi_{\varepsilon_2}(y_1) \mathrm{d} y_1\bigg)^{1/2}
			\phi_{\varepsilon_1}(x) \mathrm{d} x
						\nonumber\\
		& \,\, + X^{797 / 994}(\log X)^{-56}
			\int_{|x| \leqslant K_1}
			\bigg(\sup_{\tau_2 \leqslant|y_1| \leqslant K_2}
			\int_{\substack{\tau_2 \leqslant|y_2| \leqslant K_2 \\
				X^{-d}<|y_2-y_1| \leqslant 2 K_2}}
			|S(x, y_2)|^5 \phi_{\varepsilon_2}(y_2) \mathrm{d} y_2\bigg)^{1/2}
						\nonumber\\
		& \,\, \hspace{5em} \times
			\bigg(\int_{\tau_2 \leqslant|y_1| \leqslant K_2}
			|S(x, y_1)|^5 \phi_{\varepsilon_2}(y_1) \mathrm{d} y_1\bigg)^{1/2}
			\phi_{\varepsilon_1}(x) \mathrm{d} x
						\nonumber\\
		& \,\, + X^{1-d / 2}(\log X) \int_{|x| \leqslant K_1}
			\!\! \bigg(\sup_{\tau_2 \leqslant|y_1| \leqslant K_2}
				\int_{\substack{\tau_2 \leqslant|y_2| \leqslant K_2 \\
					X^{-d}<|y_2-y_1| \leqslant 2 K_2}}
			|S(x, y_2)|^5 |y_2-y_1|^{-1} \phi_{\varepsilon_2}(y_2)
			\mathrm{d} y_2\bigg)^{1/2}
						\nonumber\\
		& \,\, \hspace{5em} \times
			\bigg(\int_{\tau_2 \leqslant|y_1| \leqslant K_2}
			|S(x, y_1)|^5 \phi_{\varepsilon_2}(y_1) \mathrm{d} y_1\bigg)^{1/2}
			\phi_{\varepsilon_1}(x) \mathrm{d} x
						\nonumber\\
		=: & \,\, X(\log X) \cdot \mathcal{I}_1
			 + X^{797/994}(\log X)^{-56} \cdot \mathcal{I}_2
			 + X^{1-d/2}(\log X) \cdot \mathcal{I}_3.
	\end{align}
	It follows from Cauchy's inequality, Lemma \ref{S(x,y)-int-mean-2} and Lemma \ref{S(x,y)-Omega_2} that
	\begin{align}\label{eq2_15}
		\mathcal{I}_1
		\ll & \,\, \Bigg(\int_{|x|\leqslant K_1}
			\Bigg(\sup_{\tau_2\leqslant|y_1|\leqslant K_2}
			\int_{\substack{\tau_2\leqslant|y_2|\leqslant K_2 \\
					|y_2-y_1|\leqslant X^{-d}}}
			\big|S(x,y_2)\big|^5 \phi_{\varepsilon_2}(y_2)\mathrm{d}y_2\Bigg)
			\phi_{\varepsilon_1}(x)
			\mathrm{d}x\Bigg)^{1/2}
					\nonumber \\
		& \,\, \times\Bigg(\int_{|x|\leqslant K_1}
			\Bigg(\int_{\tau_2\leqslant|y_1|\leqslant K_2}
			\big|S(x,y_1)\big|^5 \phi_{\varepsilon_2}(y_1)\mathrm{d}y_1\Bigg)
			\phi_{\varepsilon_1}(x)
			\mathrm{d}x\Bigg)^{1/2}
					\nonumber \\
		\ll & \,\, \sup_{(x,y_2)\in\Omega_2}\big|S(x,y_2)\big|^{3/2}
			\times\Bigg(\sup_{\tau_2\leqslant|y_1|\leqslant K_2}
			\int_{\substack{\tau_2\leqslant|y_2|\leqslant K_2 \\
					|y_2-y_1|\leqslant X^{-d}}}
			\phi_{\varepsilon_2}(y_2)\mathrm{d}y_2
			\int_{|x|\leqslant K_1}\big|S(x,y_2)\big|^2
			\phi_{\varepsilon_1}(x)\mathrm{d}x\Bigg)^{1/2}
					\nonumber \\
		& \,\, \times\Bigg(\int_{|x|\leqslant K_1}
			\Bigg(\int_{\tau_2\leqslant|y_1|\leqslant K_2}
			\big|S(x,y_1)\big|^5 \phi_{\varepsilon_2}(y_1)\mathrm{d}y_1\Bigg)
			\phi_{\varepsilon_1}(x)
			\mathrm{d}x\Bigg)^{1/2}
					\nonumber \\
		\ll & \,\, \big(X^{66/71}(\log X)^{205}\big)^{3/2}
			\cdot \big(\varepsilon_2X^{-d}
			\cdot X(\log X)^4\big)^{1/2}
			\cdot \mathcal{J}^{1/2}
					\nonumber\\
		\ll & \,\, X^{95/71} (\log X)^{510} \cdot \mathcal{J}^{1/2},
	\end{align}
	where
	\begin{equation*}
		 \mathcal{J} :=
		 \int_{|x|\leqslant K_1}
		 \Bigg(\int_{\tau_2\leqslant|y_1|\leqslant K_2}
		 \big|S(x,y_1)\big|^5 \phi_{\varepsilon_2}(y_1)\mathrm{d}y_1\Bigg)
		 \phi_{\varepsilon_1}(x)
		 \mathrm{d}x.
	\end{equation*}
	By Cauchy's inequality, we get
	\begin{align}\label{eq2_16}
		\mathcal{I}_2
		\ll & \,\, \Bigg(\sup_{\tau_2\leqslant|y_1|\leqslant K_2}
			\int_{|x|\leqslant K_1}
			\int_{\substack{\tau_2\leqslant|y_2|\leqslant K_2 \\
					 X^{-d}<|y_2-y_1|\leqslant2K_2}}
			\big|S(x,y_2)\big|^5 \phi_{\varepsilon_1}(x)\phi_{\varepsilon_2}(y_2)
			\mathrm{d}x\mathrm{d}y_2\Bigg)^{1/2}
					\nonumber \\
		& \,\, \times\Bigg(\int_{|x|\leqslant K_1}
			\int_{\tau_2\leqslant|y_1|\leqslant K_2}
			\big|S(x,y_1)\big|^5 \phi_{\varepsilon_1}(x)\phi_{\varepsilon_2}(y_1)
			\mathrm{d}x\mathrm{d}y_1\Bigg)^{1/2} \ll  \mathcal{J}.
	\end{align}
	By Cauchy's inequality, Lemma \ref{S(x,y)-int-mean-2} and Lemma \ref{S(x,y)-Omega_2}
	again, we deduce that
	\begin{align}\label{eq2_17}
		\mathcal{I}_3
		\ll & \,\, \Bigg(\int_{|x|\leqslant K_1}
			\Bigg(\sup_{\tau_2\leqslant|y_1|\leqslant K_2}
			\int_{\substack{\tau_2\leqslant|y_2|\leqslant K_2 \\
					 X^{-d}<|y_2-y_1|\leqslant2K_2}}
			\!\!\!\big|S(x,y_2)\big|^5 |y_2-y_1|^{-1}\phi_{\varepsilon_2}(y_2)
			\mathrm{d}y_2\Bigg)
			\phi_{\varepsilon_1}(x)\mathrm{d}x\Bigg)^{1/2}
					\nonumber \\
		& \,\, \,\,\times\Bigg(\int_{|x|\leqslant K_1}
			\Bigg(\int_{\tau_2\leqslant|y_1|\leqslant K_2}
			\big|S(x,y_1)\big|^5 \phi_{\varepsilon_2}(y_1)\mathrm{d}y_1\Bigg)
			\phi_{\varepsilon_1}(x)\mathrm{d}x\Bigg)^{1/2}
					\nonumber \\
		\ll & \,\, \Bigg(\sup_{\tau_2\leqslant|y_1|\leqslant K_2}
			\int_{\substack{\tau_2\leqslant|y_2|\leqslant K_2 \\
					X^{-d}<|y_2-y_1|\leqslant2K_2}}
			|y_2-y_1|^{-1}\phi_{\varepsilon_2}(y_2)\mathrm{d}y_2
			\int_{|x|\leqslant K_1}
			\big|S(x,y_2)\big|^5 \phi_{\varepsilon_1}(x)\mathrm{d}x\Bigg)^{1/2}
					\nonumber \\
		& \,\, \,\,\times\Bigg(\int_{|x|\leqslant K_1}
			\int_{\tau_2\leqslant|y_1|\leqslant K_2}
			\big|S(x,y_1)\big|^5 \phi_{\varepsilon_1}(x)
			\phi_{\varepsilon_2}(y_1)\mathrm{d}x\mathrm{d}y_1\Bigg)^{1/2}
					\nonumber \\
		\ll & \,\, \sup_{(x,y_2)\in\Omega_2}\big|S(x,y_2)\big|^{3/2}
			\times\Bigg(\sup_{\tau_2\leqslant|y_1|\leqslant K_2}
			\int_{\substack{\tau_2\leqslant|y_2|\leqslant K_2 \\
					 X^{-d}<|y_2-y_1|\leqslant2K_2}}
			|y_2-y_1|^{-1}\phi_{\varepsilon_2}(y_2)\mathrm{d}y_2
					\nonumber \\
		& \,\, \,\,\times\int_{|x|\leqslant K_1}
			\big|S(x,y_2)\big|^2\phi_{\varepsilon_1}(x)\mathrm{d}x\Bigg)^{1/2}
			\cdot \mathcal{J}^{1/2}
				\nonumber \\
		\ll & \,\, \big(X^{66/71}(\log X)^{205}\big)^{3/2}
		\cdot\big(\varepsilon_2X(\log X)^4\big)^{1/2}
		\cdot\mathcal{J}^{1/2}
		\nonumber \\
		& \,\, \quad\times\Bigg(\sup_{\tau_2\leqslant|y_1|\leqslant K_2}
		\int_{\substack{\tau_2\leqslant|y_2|\leqslant K_2 \\ X^{-d}<|y_2-y_1|\leqslant2K_2}}
		|y_2-y_1|^{-1}\mathrm{d}y_2\Bigg)^{1/2}
		\nonumber \\
		\ll & \,\, X^{95/71+d/2} (\log X)^{511} \cdot\mathcal{J}^{1/2}.
	\end{align}
	We now turn to the estimation of $\mathcal{J}$. Recall the definition of $\mathcal{J}$, we have
	\begin{equation}\label{eq2_18}
		\mathcal{J}=\int_{|x| \leqslant K_1} \phi_{\varepsilon_1}(x) \mathrm{d} x \int_{\tau_2 \leqslant|y| \leqslant K_2}|S(x, y)|^5 \phi_{\varepsilon_2}(y) \mathrm{d} y.
	\end{equation}
	For the innermost integral on the right--hand side of (\ref{eq2_18}), it follows from the definition of $S(x, y)$ that
	\begin{align*}
		& \,\, \bigg|\int_{\tau_2 \leqslant|y| \leqslant K_2}
		|S(x, y)|^5 \phi_{\varepsilon_2}(y) \mathrm{d} y \bigg|
				\\
		= & \,\, \bigg|\int_{\tau_2 \leqslant|y| \leqslant K_2}
		S(x, y) \overline{S(x, y)} |S(x, y)|^3
		\phi_{\varepsilon_2}(y) \mathrm{d} y \bigg|
				\\
		= & \,\, \bigg|\sum_{\lambda X<p \leqslant X}(\log p) e(p^c x)
		\int_{\tau_2 \leqslant|y| \leqslant K_2}
		\overline{S(x, y)} |S(x, y)|^3
		e(p^d y) \phi_{\varepsilon_2}(y) \mathrm{d} y \bigg|
				\\
		\leqslant & \,\, (\log X) \sum_{\lambda X<p \leqslant X}
		\bigg|\int_{\tau_2 \leqslant|y| \leqslant K_2}
		\overline{S(x, y)} |S(x, y)|^3
		e(p^d y) \phi_{\varepsilon_2}(y) \mathrm{d} y \bigg|
				\\
		\leqslant & \,\, (\log X) \sum_{\lambda X<n \leqslant X}
		\bigg|\int_{\tau_2 \leqslant|y| \leqslant K_2}
		\overline{S(x, y)} |S(x, y)|^3
		e(n^d y) \phi_{\varepsilon_2}(y) \mathrm{d} y \bigg|,
	\end{align*}
	which combined with Cauchy's inequality yields that
	\begin{align}\label{eq2_19}
		& \,\, \bigg|\int_{\tau_2 \leqslant|y| \leqslant K_2}
			|S(x, y)|^5 \phi_{\varepsilon_2}(y) \mathrm{d} y\bigg|^2
					\nonumber\\
		\ll & \,\, X(\log X)^2 \sum_{\lambda X<n \leqslant X}
			\bigg|\int_{\tau_2 \leqslant|y| \leqslant K_2}
			\overline{S(x, y)}|S(x, y)|^3 e(n^d y)
			\phi_{\varepsilon_2}(y) \mathrm{d} y\bigg|^2
					\nonumber\\
		= & \,\,  X(\log X)^2 \sum_{\lambda X<n \leqslant X}
			\int_{\tau_2 \leqslant|y_1| \leqslant K_2}
			S(x, y_1)|S(x, y_1)|^3 e(-n^d y_1)
			\phi_{\varepsilon_2}(y_1) \mathrm{d} y_1
					\nonumber\\
		& \,\, \times \int_{\tau_2 \leqslant|y_2| \leqslant K_2}
			\overline{S(x, y_2)}|S(x, y_2)|^3 e(n^d y_2)
			\phi_{\varepsilon_2}(y_2) \mathrm{d} y_2
					\nonumber\\
		= & \,\, X(\log X)^2
			\int_{\tau_2 \leqslant|y_1| \leqslant K_2}
			S(x, y_1)|S(x, y_1)|^3
					\nonumber\\
		& \,\, \times \Bigg(\int_{\tau_2 \leqslant|y_2| \leqslant K_2}
			\overline{S(x, y_2)}|S(x, y_2)|^3
			\bigg(\sum_{\lambda X<n \leqslant X} e\big(n^d(y_2-y_1)\big)\bigg)
			\phi_{\varepsilon_2}(y_2) \mathrm{d} y_2 \Bigg)
			\phi_{\varepsilon_2}(y_1) \mathrm{d} y_1
					\nonumber\\
		\ll & \,\, X(\log X)^2
			\int_{\tau_2 \leqslant|y_1| \leqslant K_2}
			|S(x, y_1)|^4
			\bigg(\int_{\tau_2 \leqslant|y_2| \leqslant K_2}
			|S(x, y_2)|^4 |\mathcal{T}_d(y_2-y_1)|
			\phi_{\varepsilon_2}(y_2) \mathrm{d} y_2\bigg)
			\phi_{\varepsilon_2}(y_1) \mathrm{d} y_1
	\end{align}
	By a similar argument from (\ref{eq2_11})--(\ref{eq2_13}) applied to the exponential sum $\mathcal{T}_d(y_2 - y_1)$ in (\ref{eq2_19}), we deduce that
	\begin{align*}
			& \,\, \int_{\tau_2 \leqslant|y_1| \leqslant K_2}|S(x, y_1)|^5 \phi_{\varepsilon_2}(y_1) \mathrm{d} y_1 \\
		\ll & \,\, X(\log X)
		\bigg(\sup_{\tau_2 \leqslant|y_1| \leqslant K_2}
		\int_{\substack{\tau_2 \leqslant|y_2| \leqslant K_2 \\
				|y_2-y_1| \leqslant X^{-d}}}
		|S(x, y_2)|^4 \phi_{\varepsilon_2}(y_2) \mathrm{d} y_2\bigg)^{1/2}
		\\
		& \,\, \hspace{5em} \times
		\bigg(\int_{\tau_2 \leqslant|y_1| \leqslant K_2}
		|S(x, y_1)|^4 \phi_{\varepsilon_2}(y_1) \mathrm{d} y_1\bigg)^{1/2}
		\\
		& \,\, + X^{797/994}(\log X)^{-56}
		\bigg(\sup_{\tau_2 \leqslant|y_1| \leqslant K_2} \int_{\substack{\tau_2 \leqslant|y_2| \leqslant K_2 \\
				X^{-d}<|y_2-y_1| \leqslant 2 K_2}}
		|S(x, y_2)|^4 \phi_{\varepsilon_2}(y_2) \mathrm{d} y_2\bigg)^{1/2}
		\\
		& \,\, \hspace{5em} \times
		\bigg(\int_{\tau_2 \leqslant|y_1| \leqslant K_2}
		|S(x, y_1)|^4 \phi_{\varepsilon_2}(y_1) \mathrm{d} y_1\bigg)^{1/2}
		\\
		& \,\, + X^{1-d / 2} (\log X)
		\bigg(\sup_{\tau_2 \leqslant|y_1| \leqslant K_2} \int_{\substack{\tau_2 \leqslant|y_2| \leqslant K_2 \\
				X^{-d}<|y_2-y_1| \leqslant 2 K_2}}
		|S(x, y_2)|^4 |y_2-y_1|^{-1}
		\phi_{\varepsilon_2}(y_2) \mathrm{d} y_2\bigg)^{1/2}
		\\
		& \,\, \hspace{5em} \times
		\bigg(\int_{\tau_2 \leqslant|y_1| \leqslant K_2}
		|S(x, y_1)|^4 \phi_{\varepsilon_2}(y_1) \mathrm{d} y_1\bigg)^{1/2}
	\end{align*}
	which combined with (\ref{eq2_18}) yields that
	\begin{align}\label{eq2_20}
		\mathcal{J}
		\ll & \,\, X(\log X)
		\int_{|x| \leqslant K_1}
		\bigg(\sup_{\tau_2 \leqslant|y_1| \leqslant K_2}
		\int_{\substack{\tau_2 \leqslant|y_2| \leqslant K_2 \\
				|y_2-y_1| \leqslant X^{-d}}}
		|S(x, y_2)|^4 \phi_{\varepsilon_2}(y_2) \mathrm{d} y_2\bigg)^{1/2}
		\nonumber\\
		& \,\, \hspace{5em} \times
		\bigg(\int_{\tau_2 \leqslant|y_1| \leqslant K_2}
		|S(x, y_1)|^4 \phi_{\varepsilon_2}(y_1) \mathrm{d} y_1\bigg)^{1/2}
		\phi_{\varepsilon_1}(x) \mathrm{d} x
		\nonumber\\
		& \,\, + X^{797 / 994}(\log X)^{-56}
		\int_{|x| \leqslant K_1}
		\bigg(\sup_{\tau_2 \leqslant|y_1| \leqslant K_2}
		\int_{\substack{\tau_2 \leqslant|y_2| \leqslant K_2 \\
				X^{-d}<|y_2-y_1| \leqslant 2 K_2}}
		|S(x, y_2)|^4 \phi_{\varepsilon_2}(y_2) \mathrm{d} y_2\bigg)^{1/2}
		\nonumber\\
		& \,\, \hspace{5em} \times
		\bigg(\int_{\tau_2 \leqslant|y_1| \leqslant K_2}
		|S(x, y_1)|^4 \phi_{\varepsilon_2}(y_1) \mathrm{d} y_1\bigg)^{1/2}
		\phi_{\varepsilon_1}(x) \mathrm{d} x
		\nonumber\\
		& \,\, + X^{1-d / 2}(\log X) \int_{|x| \leqslant K_1}
		\!\! \bigg(\sup_{\tau_2 \leqslant|y_1| \leqslant K_2}
		\int_{\substack{\tau_2 \leqslant|y_2| \leqslant K_2 \\
				X^{-d}<|y_2-y_1| \leqslant 2 K_2}}
		|S(x, y_2)|^4 |y_2-y_1|^{-1} \phi_{\varepsilon_2}(y_2)
		\mathrm{d} y_2\bigg)^{1/2}
		\nonumber\\
		& \,\, \hspace{5em} \times
		\bigg(\int_{\tau_2 \leqslant|y_1| \leqslant K_2}
		|S(x, y_1)|^4 \phi_{\varepsilon_2}(y_1) \mathrm{d} y_1\bigg)^{1/2}
		\phi_{\varepsilon_1}(x) \mathrm{d} x
		\nonumber\\
		=: & \,\, X(\log X) \cdot \mathcal{I}'_1
		+ X^{797/994}(\log X)^{-56} \cdot \mathcal{I}'_2
		+ X^{1-d/2}(\log X) \cdot \mathcal{I}'_3.
	\end{align}
	It follows from Cauchy's inequality, Lemma \ref{S(x,y)-int-mean-2}, Lemma \ref{expo-full-inte} and Lemma \ref{S(x,y)-Omega_2} that
	\begin{align}\label{eq2_21}
		\mathcal{I}'_1
		\ll & \,\, \Bigg(\int_{|x|\leqslant K_1}
		\Bigg(\sup_{\tau_2\leqslant|y_1|\leqslant K_2}
		\int_{\substack{\tau_2\leqslant|y_2|\leqslant K_2 \\
				|y_2-y_1|\leqslant X^{-d}}}
		\big|S(x,y_2)\big|^4 \phi_{\varepsilon_2}(y_2)\mathrm{d}y_2\Bigg)
		\phi_{\varepsilon_1}(x)
		\mathrm{d}x\Bigg)^{1/2}
		\nonumber \\
		& \,\, \times\Bigg(\int_{|x|\leqslant K_1}
		\Bigg(\int_{\tau_2\leqslant|y_1|\leqslant K_2}
		\big|S(x,y_1)\big|^4 \phi_{\varepsilon_2}(y_1)\mathrm{d}y_1\Bigg)
		\phi_{\varepsilon_1}(x)
		\mathrm{d}x\Bigg)^{1/2}
		\nonumber \\
		\ll & \,\, \sup_{(x,y_2)\in\Omega_2}\big|S(x,y_2)\big|
		\times\Bigg(\sup_{\tau_2\leqslant|y_1|\leqslant K_2}
		\int_{\substack{\tau_2\leqslant|y_2|\leqslant K_2 \\
				|y_2-y_1|\leqslant X^{-d}}}
		\phi_{\varepsilon_2}(y_2)\mathrm{d}y_2
		\int_{|x|\leqslant K_1}\big|S(x,y_2)\big|^2
		\phi_{\varepsilon_1}(x)\mathrm{d}x\Bigg)^{1/2}
		\nonumber \\
		& \,\, \times\Bigg(\int_{|x|\leqslant K_1}
		\Bigg(\int_{\tau_2\leqslant|y_1|\leqslant K_2}
		\big|S(x,y_1)\big|^4 \phi_{\varepsilon_2}(y_1)\mathrm{d}y_1\Bigg)
		\phi_{\varepsilon_1}(x)
		\mathrm{d}x\Bigg)^{1/2}
		\nonumber \\
		\ll & \,\, X^{66/71}(\log X)^{205}
		\cdot \big(\varepsilon_2X^{-d}
		\cdot X(\log X)^4\big)^{1/2}
		\cdot \big(X^2(\log x)^6\big)^{1/2}
		\nonumber\\
		\ll & \,\, X^{133/71} (\log X)^{411} ,
	\end{align}
	By Cauchy's inequality and Lemma \ref{expo-full-inte}, we get
	\begin{align}\label{eq2_22}
		\mathcal{I}'_2
		\ll & \,\, \Bigg(\sup_{\tau_2\leqslant|y_1|\leqslant K_2}
		\int_{|x|\leqslant K_1}
		\int_{\substack{\tau_2\leqslant|y_2|\leqslant K_2 \\
				X^{-d}<|y_2-y_1|\leqslant2K_2}}
		\big|S(x,y_2)\big|^4 \phi_{\varepsilon_1}(x)\phi_{\varepsilon_2}(y_2)
		\mathrm{d}x\mathrm{d}y_2\Bigg)^{1/2}
		\nonumber \\
		& \,\, \times\Bigg(\int_{|x|\leqslant K_1}
		\int_{\tau_2\leqslant|y_1|\leqslant K_2}
		\big|S(x,y_1)\big|^4 \phi_{\varepsilon_1}(x)\phi_{\varepsilon_2}(y_1)
		\mathrm{d}x\mathrm{d}y_1\Bigg)^{1/2}
		\nonumber \\
		\ll & \,\, X^2 (\log X)^6.
	\end{align}
	By Cauchy's inequality, Lemma \ref{S(x,y)-int-mean-2} and Lemma \ref{S(x,y)-Omega_2}
	again, we deduce that
	\begin{align}\label{eq2_23}
		\mathcal{I}'_3
		\ll & \,\, \Bigg(\int_{|x|\leqslant K_1}
		\Bigg(\sup_{\tau_2\leqslant|y_1|\leqslant K_2}
		\int_{\substack{\tau_2\leqslant|y_2|\leqslant K_2 \\
				X^{-d}<|y_2-y_1|\leqslant2K_2}}
		\!\!\!\big|S(x,y_2)\big|^4 |y_2-y_1|^{-1}\phi_{\varepsilon_2}(y_2)
		\mathrm{d}y_2\Bigg)
		\phi_{\varepsilon_1}(x)\mathrm{d}x\Bigg)^{1/2}
		\nonumber \\
		& \,\, \,\,\times\Bigg(\int_{|x|\leqslant K_1}
		\Bigg(\int_{\tau_2\leqslant|y_1|\leqslant K_2}
		\big|S(x,y_1)\big|^4 \phi_{\varepsilon_2}(y_1)\mathrm{d}y_1\Bigg)
		\phi_{\varepsilon_1}(x)\mathrm{d}x\Bigg)^{1/2}
		\nonumber \\
		\ll & \,\, \Bigg(\sup_{\tau_2\leqslant|y_1|\leqslant K_2}
		\int_{\substack{\tau_2\leqslant|y_2|\leqslant K_2 \\
				X^{-d}<|y_2-y_1|\leqslant2K_2}}
		|y_2-y_1|^{-1}\phi_{\varepsilon_2}(y_2)\mathrm{d}y_2
		\int_{|x|\leqslant K_1}
		\big|S(x,y_2)\big|^4 \phi_{\varepsilon_1}(x)\mathrm{d}x\Bigg)^{1/2}
		\nonumber \\
		& \,\, \,\,\times\Bigg(\int_{|x|\leqslant K_1}
		\int_{\tau_2\leqslant|y_1|\leqslant K_2}
		\big|S(x,y_1)\big|^4 \phi_{\varepsilon_1}(x)
		\phi_{\varepsilon_2}(y_1)\mathrm{d}x\mathrm{d}y_1\Bigg)^{1/2}
		\nonumber \\
		\ll & \,\, \sup_{(x,y_2)\in\Omega_2}\big|S(x,y_2)\big|
		\times\Bigg(\sup_{\tau_2\leqslant|y_1|\leqslant K_2}
		\int_{\substack{\tau_2\leqslant|y_2|\leqslant K_2 \\
				X^{-d}<|y_2-y_1|\leqslant2K_2}}
		|y_2-y_1|^{-1}\phi_{\varepsilon_2}(y_2)\mathrm{d}y_2
		\nonumber \\
		& \,\, \hspace{10em} \times\int_{|x|\leqslant K_1}
		\big|S(x,y_2)\big|^2\phi_{\varepsilon_1}(x)\mathrm{d}x\Bigg)^{1/2}
			\nonumber \\
		& \,\, \,\,\times\Bigg(\int_{|x|\leqslant K_1}
		\Bigg(\int_{\tau_2\leqslant|y_1|\leqslant K_2}
		\big|S(x,y_1)\big|^4 \phi_{\varepsilon_2}(y_1)\mathrm{d}y_1\Bigg)
		\phi_{\varepsilon_1}(x)\mathrm{d}x\Bigg)^{1/2}
			\nonumber \\
		\ll & \,\, X^{66/71}(\log X)^{205}
		\cdot\big(\varepsilon_2X(\log X)^4\big)^{1/2}
		\cdot ({X^2 (\log X)^6})^{1/2}
		\nonumber \\
		& \,\, \,\, \times\Bigg(\sup_{\tau_2\leqslant|y_1|\leqslant K_2}
		\int_{\substack{\tau_2\leqslant|y_2|\leqslant K_2 \\ X^{-d}<|y_2-y_1|\leqslant2K_2}}
		|y_2-y_1|^{-1}\mathrm{d}y_2\Bigg)^{1/2}
		\nonumber \\
		\ll & \,\, X^{133/71+d/2} (\log X)^{411}.
	\end{align}
	From (\ref{eq2_20})--(\ref{eq2_23}), we derive that
	\begin{align*}
		\mathcal{J}
		\ll & \,\, X(\log X)\cdot X^{133/71} (\log X)^{411}
			+ X^{797/994}(\log X)^{-56}\cdot X^2(\log X)^6
					\nonumber \\
		& \,\, + X^{1-d/2}(\log X)\cdot X^{133/71+d/2} (\log X)^{411}
					\nonumber \\
		\ll & \,\, X^{204/71}(\log X)^{412}.
	\end{align*}
	which combined with (\ref{eq2_14})--(\ref{eq2_17}) yields that
	\begin{align}\label{D_2^{(1)}-upper-final}
		\mathscr{D}_2^{(1)}
		\ll & \,\, X(\log X) \cdot X^{95/71}(\log X)^{510} \cdot \big(X^{204/71}(\log X)^{412}\big)^{1/2}
					\nonumber\\
		& \,\, + X^{797/994}(\log X)^{-56} \cdot X^{204/71}(\log X)^{412}
					\nonumber \\
		& \,\, + X^{1-d/2}(\log X)\cdot X^{95/71+d/2}(\log X)^{511} \cdot \big(X^{204/71}(\log X)^{412}\big)^{1/2}
					\nonumber \\
		\ll & \,\, X^{268/71}(\log X)^{718}.
	\end{align}
	Also, by the symmetric property of the region in $\mathscr{D}_2^{(1)}$ and $\mathscr{D}_2^{(2)}$, one can follow
	the above process to deduce that
	\begin{align}\label{D_2^{(2)}-upper-final}
		\mathscr{D}_2^{(2)}\ll & \,\, X^{268/71}(\log X)^{718}.
	\end{align}
	Combining (\ref{eq2_8}), (\ref{D_2^{(1)}-upper-final}) and (\ref{D_2^{(2)}-upper-final}), we obtain
	\begin{equation*}
		\mathscr{D}_2\ll\mathscr{D}_2^{(1)}+\mathscr{D}_2^{(2)}\ll X^{268/71}(\log X)^{718}
		\ll\varepsilon_1\varepsilon_2X^{6-c-d}(\log X)^{-80}.
	\end{equation*}
	This completes the proof of Theorem \ref{Theorem-1}.

	\section{Preliminary Lemmas  }
	In this section, we shall demonstrate some lemmas, which are necessary for the proving process of Lemma \ref{S(x,y)-Omega_2}, as follows.
	\begin{lemma}\label{inverse-lemma}
		Suppose that $f(x)$ and $\varphi(x)$ are algebraic functions, which satisfy the following conditions constrained on the interval $[a,b]$:
		\begin{align*}
			& |f''(x)|\asymp R^{-1}, \qquad |f'''(x)|\ll(RU)^{-1},\qquad U\geqslant1,
			\nonumber \\
			& |\varphi(x)|\ll H,\qquad |\varphi'(x)|\ll HU_1^{-1},\qquad U_1\geqslant1.
		\end{align*}
		Then we have
		\begin{align*}
			\sum_{a<n\leqslant b}\varphi(x)e\big(f(n)\big)
			= & \,\, \sum_{\alpha<\nu\leqslant\beta}b_\nu\frac{\varphi(n_\nu)}{\sqrt{|f''(n_\nu)|}}
			e\big(f(n_\nu)-\nu n_\nu+1/8\big)
			\nonumber \\
			& \,\, +O\big(H\log(\beta-\alpha+2)+H(b-a+R)(U^{-1}+U_1^{-1})\big)
			\nonumber \\
			& \,\, +O\big(H\min\big(\sqrt{R},\max\big(\langle\alpha\rangle^{-1},\langle\beta\rangle^{-1}\big)\big)\big),
		\end{align*}
		where $[\alpha,\beta]$ is the image of $[a,b]$ under the mapping $y=f'(x)$; $n_\nu$ is the solution of the equation
		$f'(x)=\nu$;
		\begin{equation*}
			b_\nu=
			\begin{cases}
				1/2, & \textrm{if $\nu=\alpha\in\mathbb{Z}$ or $\nu=\beta\in\mathbb{Z}$},
				\nonumber \\
				\,\,\,1, &   \textrm{if $\alpha<\nu<\beta$};
			\end{cases}
		\end{equation*}
		and the function $\langle\cdot\rangle$ is defined by
		\begin{equation*}
			\langle t\rangle=
			\begin{cases}
				\,\,\, \|t\|, & \textrm{if $t\not\in\mathbb{Z}$}, \\
				\beta-\alpha, & \textrm{otherwise},
			\end{cases}
		\end{equation*}
		where $\|t\|=\min_{m\in\mathbb{Z}}|t-m|$.
	\end{lemma}
	\begin{proof}
		See Theorem 1 of Chapter III of Karatsuba and Voronin \cite{Karatsuba-Voronin-1992}.
	\end{proof}

	\begin{lemma}\label{I-K-8.17}
		Let $L>k,\,Q>0$, and $z_k$ be any complex numbers. Then we have
		\begin{equation*}
			\bigg|\sum_{K<k\leqslant L}z_k\bigg|^2\leqslant\bigg(1+\frac{L-K}{Q}\bigg)\sum_{0\leqslant|q|\leqslant Q}
			\bigg(1-\frac{|q|}{Q}\bigg)\sum_{\substack{K<k\leqslant L\\ K<k+q\leqslant L}}z_{k+q}\overline{z_k}.
		\end{equation*}
	\end{lemma}
	\begin{proof}
		See Lemma 8.17 of Iwaniec and Kowalski \cite{Iwaniec-Kowalski-book}.
	\end{proof}

	\begin{lemma}\label{error-upper-lemma}
		Suppose that $f(x)\ll P$ and $f'(x)\gg\Delta$ for $x\sim N$. Then we have
		\begin{equation*}
			\sum_{n\sim N}\min\bigg(D,\frac{1}{\|f(n)\|}\bigg)\ll(P+1)(D+\Delta^{-1})\log(2+\Delta^{-1}).
		\end{equation*}
	\end{lemma}
	\begin{proof}
		See Lemma 2.8 of Kr\"{a}tzel \cite{Kratzel-book}.
	\end{proof}

	\begin{lemma}\label{van-der-Corput-lemma}
		Suppose that $5<A<B\leqslant 2A$ and $f''(x)$ is continuous on $[A,B]$. If $0<c_1\lambda_1\leqslant|f'(x)|\leqslant c_2\lambda_1\leqslant1/2$, then
		\begin{equation*}
			\sum_{A<n\leqslant B}e(f(n)))\ll\lambda_1^{-1}.
		\end{equation*}
		If $0<c_3\lambda_2\leqslant|f''(x)|\leqslant c_4\lambda_2$, then
		\begin{equation*}
			\sum_{A<n\leqslant B}e(f(n)))\ll A\lambda_2^{1/2}+\lambda_{2}^{-1/2}.
		\end{equation*}
	\end{lemma}
	\begin{proof}
		See Theorem 2.1 of Graham and Kolesnik \cite{Graham-Kolesnik-book} and Corollary 8.13 of Iwaniec and Kowalski \cite{Iwaniec-Kowalski-book}.
	\end{proof}

	\begin{lemma}
		Let $M$ and $M_1$ be positive numbers such that $5\leqslant M<M_1\leqslant2M$. Suppose that $a$ and $b$ are real numbers subject to $ab\not=0$, and $\gamma_1$ and $\gamma_2$ are two distinct real numbers with $1<\gamma_1,\gamma_2<2$. Define
		\begin{equation*}
			\mathcal{S}:=\mathcal{S}(M,a,b,\gamma_1,\gamma_2)=\sum_{M<m\leqslant M_1}e(am^{\gamma_1}+bm^{\gamma_2}),
		\end{equation*}
		and
		\begin{equation*}
			\mathcal{R}=|a|M^{\gamma_1}+|b|M^{\gamma_2}.
		\end{equation*}
		If $\mathcal{R}M^{-1}\leqslant1/8$, then one has
		\begin{equation}\label{S-upper-1}
			\mathcal{S}\ll M\mathcal{R}^{-1/2}.
		\end{equation}
		If $M\ll\mathcal{R}\ll M^2$, then one has
		\begin{equation}\label{S-upper-2}
			\mathcal{S}\ll \mathcal{R}^{1/2}+M\mathcal{R}^{-1/3}.
		\end{equation}
	\end{lemma}
	\begin{proof}
		See Lemma 5 and Lemma 6 of Zhai \cite{Zhai-2000}.
	\end{proof}

	Now, we focus on the upper bound estimate of $S(x,y)$ for $(x,y)\in\Omega_2$. Suppose that $1<d<c<79/71$ and fix
	$(x,y)\in\Omega_2$. Write $R=|x|X^c+|y|X^d$. Trivially, there holds
	$X^{3/4-\eta}\ll R\ll X^{79/71}(\log X)^{-400}$.
	
	\begin{lemma}\label{Type-I-upper}
		Suppose that $a(m)$ are sequences supported on the intervals $(M,2M]$. Suppose further that
		\begin{equation*}
			\sum_{m\sim M}\big|a(m)\big|^2\ll M(\log M)^{2A},\qquad A>0.
		\end{equation*}
		Then, for $M\ll\min\big(X^{51/71},X^{117/71}R^{-1}\big),\,MN\asymp X$, we have
		\begin{equation*}
			S_I=\sum_{m\sim M}a(m)\sum_{n\sim N}e\big(x(mn)^c+y(mn)^d\big)\ll X^{66/71}(\log X)^{A+1}.
		\end{equation*}
	\end{lemma}
	\begin{proof}
		If $M\ll X^{66/71}R^{-1/2}$, then it follows from Cauchy's inequality and (\ref{S-upper-2}) that
		\begin{align*}
			S_I \ll & \,\, \Bigg(\sum_{m\sim M}\big|a(m)\big|^2\Bigg)^{1/2}
			\Bigg(\sum_{m\sim M}\Bigg|\sum_{n\sim N}e\big(x(mn)^c+y(mn)^d\big)\Bigg|^2\Bigg)^{1/2}
			\nonumber \\
			\ll & \,\, \big(M(\log M)^{2A}\big)^{1/2}\big(M(R^{1/2}+NR^{-1/3})^2\big)^{1/2}
			\nonumber \\
			\ll & \,\, M(R^{1/2}+NR^{-1/3})(\log M)^A\ll X^{66/71}(\log X)^A.
		\end{align*}
		From now on we always postulate that $M\gg X^{66/71}R^{-1/2}$. Set $Q=[X^{10/71}]$. It follows from Cauchy's inequality and Lemma \ref{I-K-8.17} that
		\begin{equation*}
			\big|S_I\big|^2\ll X^2Q^{-1}(\log X)^{2A}+XQ^{-1}(\log X)^{2A}\sum_{1\leqslant q\leqslant Q}\big|E_q\big|,
		\end{equation*}
		where
		\begin{equation*}
			E_q=\sum_{m\sim M}\sum_{N<n\leqslant2N-q}e\big(xm^c\Delta(n,q;c)+ym^d\Delta(n,q;d)\big),
		\end{equation*}
		and
		\begin{equation*}
			\Delta(n,q;t)=(n+q)^t-n^t.
		\end{equation*}
		Therefore, it suffices to show that
		\begin{equation*}
			\sum_{1\leqslant q\leqslant Q}\big|E_q\big|\ll X(\log X)^2.
		\end{equation*}
		For each fixed $q$ subject to $1\leqslant q\leqslant Q$, define
		\begin{equation*}
			f(m,n)=xm^c\Delta(n,q;c)+ym^d\Delta(n,q;d).
		\end{equation*}
		Now, we consider the following four cases.
		
		\noindent
		\textbf{Case 1.} If $|f_m|\leqslant10^{-3}$, then one can see easily that
		\begin{align*}
			& \big|xm^c\Delta(n,q;c)\big|\asymp q|x|m^cn^{c-1} \asymp q|x|X^cN^{-1},
			\nonumber \\
			& \big|ym^d\Delta(n,q;d)\big|\asymp q|y|m^dn^{d-1} \asymp q|y|X^dN^{-1},
		\end{align*}
		and thus
		\begin{equation*}
			\big|xm^c\Delta(n,q;c)\big|+\big|ym^d\Delta(n,q;d)\big|\asymp qRN^{-1}.
		\end{equation*}
		We apply (\ref{S-upper-1}) to the sum over $m$ to derive that
		\begin{equation*}
			E_q\ll NM(qRN^{-1})^{-1/2}\ll MN^{3/2}q^{-1/2}R^{-1/2},
		\end{equation*}
		which combined with the two conditions $M\gg X^{66/71}R^{-1/2}$ and $R\gg X^{3/4-\eta}$ yields that
		\begin{align*}
			\sum_{1\leqslant q\leqslant Q}\big|E_q\big|
			\ll & \,\, MN^{3/2}Q^{1/2}R^{-1/2}\ll X^{223/142}M^{-1/2}R^{-1/2}
			\nonumber \\
			\ll & \,\, X^{157/142}R^{-1/4}\ll X^{1043/1136+\eta}=o(X).
		\end{align*}
		
		\noindent
		\textbf{Case 2.} If $|f_n|\leqslant10^{-3}$, then for fixed $m$, there holds
		\begin{align*}
			&  f_n=cxm^c\Delta(n,q;c-1)+dym^d\Delta(n,q;d-1),
			\nonumber \\
			& \Delta(n,q;c-1)=(c-1)qn^{c-2}+O(q^2N^{c-3}),
			\nonumber \\
			& \Delta(n,q;d-1)=(d-1)qn^{d-2}+O(q^2N^{d-3}),
		\end{align*}
		and hence
		\begin{equation*}
			f_n=c(c-1)xqm^cn^{c-2}+d(d-1)yqm^dn^{d-2}+O(q^2RN^{-3}).
		\end{equation*}
		If $xy>0$, then $f_n\asymp qRN^{-2}=o(1)$, which combined with Lemma \ref{van-der-Corput-lemma} yields that
		\begin{equation*}
			E_q\ll MN^2q^{-1}R^{-1},
		\end{equation*}
		and thus
		\begin{align*}
			\sum_{1\leqslant q\leqslant Q}\big|E_q\big|
			\ll & \,\, MN^2R^{-1}\log X \ll X^2M^{-1}R^{-1}\log X
			\nonumber \\
			\ll & \,\, X^{76/71}R^{-1/2}\log X\ll X^{395/568+\eta}=o(X).
		\end{align*}
		If $xy<0$, define
		\begin{align*}
			\mathcal{J}_1=\big\{n\in(N,2N-q]:\,|f_n|\leqslant q^{1/2}R^{1/2}N^{-3/2}\big\},
			\nonumber \\
			\mathcal{J}_2=\big\{n\in(N,2N-q]:\,|f_n|> q^{1/2}R^{1/2}N^{-3/2}\big\}.
		\end{align*}
		For $n\in\mathcal{J}_1$, we get
		\begin{align*}
			c(c-1)xqm^cn^{c-2}
			= & \,\, -d(d-1)yqm^dn^{d-2}+O\big(q^{1/2}R^{1/2}N^{-3/2}+q^2RN^{-3}\big)
			\nonumber \\
			= & \,\, -d(d-1)yqm^dn^{d-2}\big(1+O\big(q^{-1/2}R^{-1/2}N^{1/2}+qN^{-1}\big)\big),
		\end{align*}
		which implies
		\begin{align*}
			n = & \,\, \bigg(-\frac{d(d-1)ym^d}{c(c-1)xm^c}\bigg)^{1/(c-d)}
			\big(1+O\big(q^{-1/2}R^{-1/2}N^{1/2}+qN^{-1}\big)\big)^{1/(c-d)}
			\nonumber \\
			= & \,\, \bigg(-\frac{d(d-1)ym^d}{c(c-1)xm^c}\bigg)^{1/(c-d)}
			\big(1+O\big(q^{-1/2}R^{-1/2}N^{1/2}+qN^{-1}\big)\big)
			\nonumber \\
			= & \,\, \bigg(-\frac{d(d-1)ym^d}{c(c-1)xm^c}\bigg)^{1/(c-d)}+O\big(q+q^{-1/2}R^{-1/2}N^{3/2}\big).
		\end{align*}
		Consequently, we obtain
		\begin{equation*}
			|\mathcal{J}_1|\ll q+q^{-1/2}R^{-1/2}N^{3/2}.
		\end{equation*}
		For $n\in\mathcal{J}_2$, it follows from Lemma \ref{van-der-Corput-lemma} that
		\begin{equation*}
			\sum_{\substack{N<n\leqslant 2N-q\\ q^{1/2}R^{1/2}N^{-3/2}<|f_n|\leqslant10^{-3}}}e(f(m,n))
			\ll q^{-1/2}R^{-1/2}N^{3/2}.
		\end{equation*}
		Therefore, for $xy<0$, one has
		\begin{equation*}
			\sum_{\substack{N<n\leqslant2N-q\\ |f_n|\leqslant10^{-3}}}e(f(m,n))\ll q+q^{-1/2}R^{-1/2}N^{3/2}.
		\end{equation*}
		Combining the above two cases (i.e., $xy>0$ and $xy<0$), we derive that
		\begin{equation*}
			\sum_{m\sim M}\sum_{\substack{N<n\leqslant 2N-q\\ |f_n|\leqslant10^{-3}}}e(f(m,n))
			\ll Mq+q^{-1/2}R^{-1/2}MN^{3/2}+MN^2q^{-1}R^{-1},
		\end{equation*}
		and thus
		\begin{align*}
			\sum_{1\leqslant q\leqslant Q}\sum_{m\sim M}\sum_{\substack{N<n\leqslant2N-q\\ |f_n|\leqslant10^{-3}}}e(f(m,n))
			\ll  & \,\, MQ^2+Q^{1/2}R^{-1/2}MN^{3/2}+MN^2R^{-1}\log Q
			\nonumber \\
			\ll & \,\, X\log X,
		\end{align*}
		provided that $X^{66/71}R^{-1/2}\ll M\ll X^{51/71}$.
		
		\noindent
		\textbf{Case 3.} Suppose that there exist two non--negative integers $i$ and $j$ subject to $2\leqslant i+j\leqslant3$ such that
		\begin{equation}\label{case-3-condition}
			\bigg|\frac{\partial^{i+j}f}{\partial m^i\partial n^j}\bigg|\leqslant\frac{qR\log X}{QM^iN^{j+1}}.
		\end{equation}
		Set
		\begin{equation*}
			\mathfrak{c}(\gamma,i)=
			\begin{cases}
				\gamma(\gamma-1)\cdots(\gamma-i+1), & \textrm{if $i\geqslant1$},  \\
				\qquad\qquad  1, & \textrm{if $i=0$}.
			\end{cases}
		\end{equation*}
		Then one has
		\begin{equation*}
			\frac{\partial^{i+j}f}{\partial m^i\partial n^j}=\mathfrak{c}(c,i)\mathfrak{c}(c,j)xm^{c-i}\Delta(n,q;c-j)
			+\mathfrak{c}(d,i)\mathfrak{c}(d,j)ym^{d-i}\Delta(n,q;d-j).
		\end{equation*}
		By noting the fact that $\mathfrak{c}(c,i)\mathfrak{c}(c,j)$ and $\mathfrak{c}(d,i)\mathfrak{c}(d,j)$ are the same sign, without loss of generality, we can postulate $xy<0$. Otherwise, there exists no pair $(m,n)$ which satisfies
		(\ref{case-3-condition}). For $(m,n)$ satisfying (\ref{case-3-condition}), we deduce that
		\begin{align*}
			\mathfrak{c}(c,i)\mathfrak{c}(c,j)xm^{c-i}\Delta(n,q;c-j)
			= & \,\, -\mathfrak{c}(d,i)\mathfrak{c}(d,j)ym^{d-i}\Delta(n,q;d-j)
			+O\bigg(\frac{qR\log X}{QM^iN^{j+1}}\bigg)
			\nonumber \\
			= & \,\, -\mathfrak{c}(d,i)\mathfrak{c}(d,j)ym^{d-i}\Delta(n,q;d-j)
			\big(1+O(Q^{-1}\log X)\big),
		\end{align*}
		which implies that
		\begin{align*}
			m = & \,\, \bigg(-\frac{\mathfrak{c}(d,i)\mathfrak{c}(d,j)y\Delta(n,q;d-j)}
			{\mathfrak{c}(c,i)\mathfrak{c}(c,j)x\Delta(n,q;c-j)}\bigg)^{1/(c-d)}
			\big(1+O(Q^{-1}\log X)\big)^{1/(c-d)}
			\nonumber \\
			= & \,\, \bigg(-\frac{\mathfrak{c}(d,i)\mathfrak{c}(d,j)y\Delta(n,q;d-j)}
			{\mathfrak{c}(c,i)\mathfrak{c}(c,j)x\Delta(n,q;c-j)}\bigg)^{1/(c-d)}
			\big(1+O(Q^{-1}\log X)\big)
			\nonumber \\
			= & \,\, \bigg(-\frac{\mathfrak{c}(d,i)\mathfrak{c}(d,j)y\Delta(n,q;d-j)}
			{\mathfrak{c}(c,i)\mathfrak{c}(c,j)x\Delta(n,q;c-j)}\bigg)^{1/(c-d)}
			+O\big(MQ^{-1}\log X\big).
		\end{align*}
		Therefore, we get
		\begin{equation*}
			\mathop{\sum_{m\sim M}\sum_{N<n\leqslant 2N-q}}_{\substack{\big|\frac{\partial^{i+j}f}{\partial m^i\partial n^j}\big|
					\leqslant\frac{qR\log X}{QM^iN^{j+1}}\\ 2\leqslant i+j\leqslant3   } }e(f(m,n))\ll\frac{MN\log X}{Q}
			\ll\frac{X\log X}{Q},
		\end{equation*}
		which implies that
		\begin{equation*}
			\sum_{1\leqslant q\leqslant Q}\mathop{\sum_{m\sim M}\sum_{N<n\leqslant 2N-q}}_{\substack{\big|\frac{\partial^{i+j}f}{\partial m^i\partial n^j}\big|
					\leqslant\frac{qR\log X}{QM^iN^{j+1}}\\ 2\leqslant i+j\leqslant3   } }e(f(m,n))\ll X\log X.
		\end{equation*}
		
		\noindent
		\textbf{Case 4.} In this case, we postulate that all the conditions which are in the Cases 1--3 do not hold. Then, $|f_n|>10^{-3}>0$, which means $f_n$ keep the same sign. Without loss of generality, we suppose that $f_n>0$.
		For any fixed $j$ subject to $0\leqslant j\leqslant(\log(10Q))/\log2$, denote by $I_j=[A_j,B_j]$ the subinterval of
		$[N,2N-q]$ in which
		\begin{equation*}
			\frac{2^jqR}{QN^3}<\bigg|\frac{\partial^2f}{\partial n^2}\bigg|\leqslant\frac{2^{j+1}qR}{QN^3},
		\end{equation*}
		where $A_j$ and $B_j$ may depend on $m$. By Lemma \ref{inverse-lemma}, we obtain
		\begin{equation*}
			\sum_{n\in I_j}e(f(m,n))=e(1/8)\sum_{\nu_1(m)<\nu\leqslant\nu_2(m)}b_\nu
			\frac{e(\mathfrak{s}(m,\nu))}{\sqrt{|\mathcal{G}(m,\nu)|}}+O\big(\mathcal{R}(m,q,j)\big),
		\end{equation*}
		where
		\begin{equation*}
			f_n(m,g(m,\nu))=\nu,\qquad\mathfrak{s}(m,\nu)=f(m,g(m,\nu))-\nu\cdot g(m,\nu),
		\end{equation*}
		\begin{equation*}
			\mathcal{G}(m,\nu)=f_{nn}(m,g(m,\nu)),\qquad \frac{qR}{QN^2}\ll\nu_1(m),\nu_2(m)\ll\frac{qR}{N^2},
		\end{equation*}
		\begin{equation*}
			\mathcal{R}(m,q,j)=\log X+\frac{QN^2}{2^jqR}+\min\bigg(\frac{Q^{1/2}N^{3/2}}{2^{j/2}q^{1/2}R^{1/2}},
			\frac{1}{\|\nu_1(m)\|},\frac{1}{\|\nu_2(m)\|}\bigg).
		\end{equation*}
		Since the condition in Case 2 does not hold, one has $|f_n|\asymp qRN^{-2}\gg1$. Moreover, by noting that
		\begin{equation*}
			\nu_1'(m)\gg\frac{qR}{QMN^2} \qquad \textrm{and}  \qquad\nu_2'(m)\gg\frac{qR}{QMN^2},
		\end{equation*}
		it follows from Lemma \ref{error-upper-lemma} that
		\begin{align*}
			& \,\, \sum_{1\leqslant q\leqslant Q}\sum_{0\leqslant j\ll\log Q}\sum_{m\sim M}\mathcal{R}(m,q,j)
			\nonumber \\
			\ll & \,\, \sum_{1\leqslant q\leqslant Q}\sum_{0\leqslant j\ll\log Q}\Bigg(M\log X+\frac{QMN^2}{2^jqR}
			+\frac{qR}{N^2}\cdot\frac{Q^{1/2}N^{3/2}}{2^{j/2}q^{1/2}R^{1/2}}+\frac{qR}{N^2}
			\cdot\frac{QMN^2}{qR}\Bigg)
			\nonumber \\
			\ll & \,\, MQ^2(\log X)^2+QMN^2R^{-1}\log X+Q^2R^{1/2}N^{-1/2}
			\nonumber \\
			\ll & \,\, X(\log X)^2.
		\end{align*}
		Set $\nu_1=\min\limits_{m\sim M}\nu_1(m),\nu_2=\max\limits_{m\sim M}\nu_2(m)$. Then we derive that
		\begin{equation*}
			\sum_{m\sim M}\sum_{\nu_1(m)<\nu\leqslant\nu_2(m)}b_\nu\frac{e(\mathfrak{s}(m,\nu))}{\sqrt{|\mathcal{G}(m,\nu)|}}
			\ll\sum_{\nu_1\leqslant\nu\leqslant\nu_2}\Bigg|\sum_{m\in\mathcal{I}_\nu}
			\frac{e(\mathfrak{s}(m,\nu))}{\sqrt{|\mathcal{G}(m,\nu)|}}\Bigg|,
		\end{equation*}
		where $\mathcal{I}_\nu$ is a subinterval of $[M,2M]$. Thus, it suffices to estimate the inner sum over $m$. First,
		we shall show that $|\mathcal{G}(m,\nu)|^{-1/2}$ is monotonic in $m$. Write $\mathfrak{g}=g(m,\nu)$ for abbreviation. Differentiating the equation $f_n(m,g(m,\nu))=\nu$ over $m$, we obtain
		\begin{equation*}
			g_m(m,\nu)=-\frac{f_{nm}(m,\mathfrak{g})}{f_{nn}(m,\mathfrak{g})},
		\end{equation*}
		and thus
		\begin{equation*}
			\mathcal{G}_m(m,\nu)=f_{mnn}+f_{nnn}g_m=\frac{f_{nnm}f_{nn}-f_{nnn}f_{nm}}{f_{nn}}.
		\end{equation*}
		By the assumption, we know that $f_{nn}$ always keeps the definite sign. Hence, it suffices to show that $f_{nnm}f_{nn}-f_{nnn}f_{nm}$ has the same sign on some subinterval of $[M,2M]$. Now, by simple calculations,
		we get
		\begin{align*}
			f_{nm} = & \,\, c^2xm^{c-1}\Delta(\mathfrak{g},q;c-1)+d^2ym^{d-1}\Delta(\mathfrak{g},q;d-1)
			\nonumber \\
			= & \,\, c^2(c-1)xqm^{c-1}\mathfrak{g}^{c-2}+d^2(d-1)yqm^{d-1}\mathfrak{g}^{d-2}
			+O\big(q^2RM^{-1}N^{-3}\big).
		\end{align*}
		Since the condition of Case 3 does not hold, we have
		\begin{equation*}
			|f_{nm}|>\frac{qR\log X}{QMN^2},
		\end{equation*}
		and thus
		\begin{equation*}
			f_{nm}=\big(c^2(c-1)xqm^{c-1}\mathfrak{g}^{c-2}+d^2(d-1)yqm^{d-1}\mathfrak{g}^{d-2}\big)
			\big(1+O\big(Q^2N^{-1}(\log X)^{-1}\big)\big).
		\end{equation*}
		For $f_{nn}$, we have
		\begin{align*}
			f_{nn} = & \,\, c(c-1)xm^c\Delta(\mathfrak{g},q;c-2)+d(d-1)ym^d\Delta(\mathfrak{g},q;d-2)
			\nonumber \\
			= & \,\, c(c-1)(c-2)qxm^c\mathfrak{g}^{c-3}+d(d-1)(d-2)qym^d\mathfrak{g}^{d-3}+O\big(q^2RN^{-4}\big).
		\end{align*}
		Since the condition of Case 3 does not hold, we have
		\begin{equation*}
			|f_{nn}|>\frac{qR\log X}{QN^3},
		\end{equation*}
		and thus
		\begin{equation*}
			f_{nn}=\big(c(c-1)(c-2)xqm^{c}\mathfrak{g}^{c-3}+d(d-1)(d-2)yqm^{d}\mathfrak{g}^{d-3}\big)
			\big(1+O\big(Q^2N^{-1}(\log X)^{-1}\big)\big).
		\end{equation*}
		In a similar way, one can deduce that
		\begin{equation*}
			f_{nnm}=\big(c^2(c-1)(c-2)xqm^{c-1}\mathfrak{g}^{c-3}+d^2(d-1)(d-2)yqm^{d-1}\mathfrak{g}^{d-3}\big)
			\big(1+O\big(Q^2N^{-1}(\log X)^{-1}\big)\big),
		\end{equation*}
		and
		\begin{equation*}
			f_{nnn}=\big(c(c-1)(c-2)(c-3)xqm^{c}\mathfrak{g}^{c-4}+d(d-1)(d-2)(d-3)yqm^{d}\mathfrak{g}^{d-4}\big)
			\big(1+O\big(Q^2N^{-1}(\log X)^{-1}\big)\big).
		\end{equation*}
		For simplicity, we set $s=xm^c\mathfrak{g}^c,t=ym^d\mathfrak{g}^d$. Under these symbols, we obtain
		\begin{align}\label{fz}
			f_{nn}f_{nnm}-f_{nm}f_{nnn}=q^2m^{-1}\mathfrak{g}^{-6}(As^2+2Bst+Ct^2)\big(1+O(Q^2N^{-1}(\log X)^{-1})\big),
		\end{align}
		where
		\begin{align*}
			A = & \,\, c^3(c-1)^2(c-2)<0,
			\nonumber \\
			B = & \,\, c(c-1)d(d-1)(3cd-c^2-d^2-c-d)<0,
			\nonumber \\
			C = & \,\, d^3(d-1)^2(d-2)<0.
		\end{align*}
		Accordingly, it suffices to show that
		\begin{equation}\label{case-4-aim}
			As^2+2Bst+Ct^2\not=0.
		\end{equation}
		If $xy>0$, then (\ref{case-4-aim}) is trivial. Now, we postulate $xy<0$. It is easy to see that
		\begin{equation*}
			4B^2-4AC=4c^2(c-1)^2d^2(d-1)^2(c-d)^2(2c+2d+1+c^2+d^2-4cd)>0,
		\end{equation*}
		which implies that there exist constants $a_1,a_2,b_1,b_2$ such that
		\begin{equation*}
			As^2+2Bst+Ct^2=(a_1s+b_1t)(a_2s+b_2t).
		\end{equation*}
		By noting the fact that $A<0,B<0,C<0$, there must hold $a_1b_1>0,a_2b_2>0$. Without loss of generality, we assume that $a_1>0,b_1>0,a_2>0,b_2>0$. For $\alpha>0$ and $\beta>0$, let $\mathscr{N}(\alpha,\beta)$ denote the number of solutions of the following inequality
		\begin{equation}\label{special-inequality}
			\big|\alpha xm^cn^c+\beta ym^dn^d\big|\leqslant\frac{R}{Q^{1/2}\log X}
		\end{equation}
		subject to $m\sim M$ and $n\sim N$. Then we claim that, for sufficiently small $\sigma\in(0,1)$, there holds $\mathscr{N}(\alpha,\beta)\ll_\sigma X^{66/71}$ uniformly for $\alpha,\beta\in[\sigma,\sigma^{-1}]$. Actually, if $xy>0$, then $\mathscr{N}(\alpha,\beta)=0$. Thus, we suppose that $xy<0$. If $(m,n)$ satisfies the inequality (\ref{special-inequality}), then
		\begin{align*}
			\alpha xm^cn^c
			= & \, -\beta ym^dn^d+O\bigg(\frac{R}{Q^{1/2}\log X}\bigg)
			\nonumber \\
			= & \, -\beta ym^dn^d\big(1+O\big(Q^{-1/2}(\log X)^{-1}\big)\big),
		\end{align*}
		which implies that
		\begin{align*}
			mn = & \,\, \bigg(-\frac{\beta y}{\alpha x}\bigg)^{1/(c-d)}
			\big(1+O\big(Q^{-1/2}(\log X)^{-1}\big)\big)^{1/(c-d)}
			\nonumber \\
			= & \,\, \bigg(-\frac{\beta y}{\alpha x}\bigg)^{1/(c-d)}
			\big(1+O\big(Q^{-1/2}(\log X)^{-1}\big)\big)
			\nonumber \\
			= & \,\, \bigg(-\frac{\beta y}{\alpha x}\bigg)^{1/(c-d)}+O\big(XQ^{-1/2}(\log X)^{-1}\big).
		\end{align*}
		Therefore, it follows from a divisor argument that
		\begin{equation}\label{N-counting-upper}
			\mathscr{N}(\alpha,\beta)\ll XQ^{-1/2}\ll X^{66/71}
		\end{equation}
		holds uniformly for $\alpha,\beta\in[\sigma,\sigma^{-1}]$ with sufficiently small $\sigma\in(0,1)$. Now, we take
		\begin{equation*}
			\sigma=\frac{1}{2}\min\big(|a_1|,|a_2|,|b_1|^{-1},|b_2|^{-1}\big).
		\end{equation*}
		According to the above arguments, without loss of generality, we postulate $s$ and $t$ do not satisfy the inequality (\ref{special-inequality}). (Otherwise, if $s$ and $t$ satisfy the inequality (\ref{special-inequality}), then we can use trivial estimate and (\ref{N-counting-upper}) to counting the number of solutions pairs $(m,n)$ to get $S_I\ll X^{66/71}$.) At this time, we have
		\begin{equation*}
			|a_1s+b_1t|>\frac{R}{Q^{1/2}\log X},\qquad |a_2s+b_2t|>\frac{R}{Q^{1/2}\log X},
		\end{equation*}
		and thus
		\begin{equation*}
			|As^2+2Bst+Ct^2|\geqslant\frac{R^2}{Q(\log X)^2},
		\end{equation*}
		which combined with (\ref{fz}) implies that $f_{nnm}f_{nn}-f_{nnn}f_{nm}$ always has the same sign on subinterval of $(M,2M]$. By the above arguments, we know that $|\mathcal{G}(m,\nu)|$ is monotonic in $m$, and so is $|\mathcal{G}(m,\nu)|^{-1/2}$. Next, we compute $\mathfrak{s}_{mm}(m,\nu)$. By simple calculation, one has
		\begin{align*}
			\mathfrak{s}_{m}(m,\nu)
			= & \,\, f_m(m,\mathfrak{g})+f_n(m,\mathfrak{g})\mathfrak{g}_m-\nu\mathfrak{g}_m=f_m(m,\mathfrak{g}),
			\nonumber \\
			\mathfrak{s}_{mm}(m,\nu)
			= & \,\, f_{mm}(m,\mathfrak{g})+f_{mn}(m,\mathfrak{g})\mathfrak{g}_m=(f_{mm}f_{nn}-f_{mn}^2)f_{nn}^{-1}.
		\end{align*}
		By the calculation of $\mathcal{G}_m$, one has
		\begin{equation*}
			f_{mm}f_{nn}-f_{mn}^2=-2q^2m^{-2}n^{-4}(A_1s^2+B_1st+C_1t^2)\big(1+O(Q^2N^{-1}(\log X)^{-1})\big),
		\end{equation*}
		where
		\begin{equation*}
			A_1=c^3(c-1)^2,\qquad B_1=c(c-1)d(d-1)(c+d),\qquad C_1=d^3(d-1)^2,
		\end{equation*}
		and $B_1^2-4A_1C_1>0$. Now, if $xy>0$, one immediately obtains
		\begin{equation*}
			\big|f_{mm}f_{nn}-f_{mn}^2\big|\gg\frac{q^2R^2}{M^2N^4}.
		\end{equation*}
		If $xy<0$, then it follows from the similar arguments of $\mathcal{G}_m$ that
		\begin{equation*}
			|A_1s^2+B_1st+C_1t^2|\gg\frac{R^2}{Q(\log X)^2},
		\end{equation*}
		which implies that
		\begin{equation*}
			\big|f_{mm}f_{nn}-f_{mn}^2\big|\gg\frac{q^2R^2}{QM^2N^4(\log X)^2}.
		\end{equation*}
		Combining the above arguments and the estimate $|f_{nn}|\asymp qRN^{-3}$, we get
		\begin{equation*}
			|\mathfrak{s}_{mm}|\gg\frac{q^2R^2}{QM^2N^4(\log X)^2}\cdot\frac{N^3}{qR}=\frac{qR}{QM^2N(\log X)^2},
		\end{equation*}
		which implies that
		\begin{equation*}
			|\mathfrak{s}_{mm}|\geqslant\frac{\mathcal{C}^*qR}{QM^2N(\log X)^2}
		\end{equation*}
		holds for some $\mathcal{C}^*>0$. On the other hand, we have
		\begin{equation*}
			|\mathfrak{g}_m|=\frac{|f_{nm}|}{|f_{nn}|}\asymp\frac{qM^{-1}N^{-2}R}{qRN^{-3}}\asymp\frac{N}{M},
		\end{equation*}
		which implies the trivial estimate
		\begin{equation*}
			\big|\mathfrak{s}_{mm}\big|\ll\big|f_{mm}\big|+\big|f_{mn}\mathfrak{g}_m\big|
			\ll\frac{qR}{M^2N}+\frac{qR}{MN^2}\cdot\frac{N}{M}\ll\frac{qR}{M^2N}.
		\end{equation*}
		For $0\leqslant\ell\leqslant(\log(Q(\log X)^2))/\log 2$, we set
		\begin{equation*}
			\mathcal{I}_{\nu,\ell}=\bigg\{m\in\mathcal{I}_{\nu}:\,\frac{\mathcal{C}^*2^\ell qR}{QM^2N(\log X)^2}<
			\big|\mathfrak{s}_{mm}\big|\leqslant\frac{\mathcal{C}^*2^{\ell+1}qR}{QM^2N(\log X)^2}\bigg\}.
		\end{equation*}
		By the assumption (i.e., the condition in Case 3 does not hold), we know that
		\begin{equation*}
			\big|\mathcal{G}(m,\nu)\big|=\big|f_{nn}(m,\mathfrak{g})\big|>\frac{qR\log X}{QN^3}>\frac{qR}{QN^3}.
		\end{equation*}
		Then it follows from partial summation and Lemma \ref{van-der-Corput-lemma} that
		\begin{align*}
			& \,\, \sum_{1\leqslant q\leqslant Q}\sum_{0\leqslant j\leqslant\frac{\log(10Q)}{\log 2}}
			\sum_{\nu_1\leqslant\nu\leqslant\nu_2}\Bigg|\sum_{m\in \mathcal{I}_\nu}
			\frac{e(\mathfrak{s}(m,\nu))}{\sqrt{|\mathcal{G}(m,\nu)|}}\Bigg|
			\nonumber \\
			\ll & \,\, \sum_{1\leqslant q\leqslant Q}\sum_{0\leqslant j\leqslant\frac{\log(10Q)}{\log 2}}
			\sum_{\nu_1\leqslant\nu\leqslant\nu_2}\sum_{0\leqslant\ell\leqslant\frac{\log(Q(\log X)^2)}{\log2}}
			\Bigg|\sum_{m\in \mathcal{I}_{\nu,\ell}}
			\frac{e(\mathfrak{s}(m,\nu))}{\sqrt{|\mathcal{G}(m,\nu)|}}\Bigg|
			\nonumber \\
			\ll & \,\, \sum_{1\leqslant q\leqslant Q}\sum_{0\leqslant j\leqslant\frac{\log(10Q)}{\log 2}}
			\sum_{\nu_1\leqslant\nu\leqslant\nu_2}\sum_{0\leqslant\ell\leqslant\frac{\log(Q(\log X)^2)}{\log2}}
			\bigg(\frac{QN^3}{qR}\bigg)^{1/2}
			\nonumber \\
			& \,\, \quad\times \Bigg(M\bigg(\frac{2^\ell qR}{QM^2N(\log X)^2}\bigg)^{1/2}
			+\bigg(\frac{QM^2N(\log X)^2}{2^\ell qR}\bigg)^{1/2}\Bigg)
			\nonumber \\
			\ll & \,\, \sum_{1\leqslant q\leqslant Q}\sum_{0\leqslant j\leqslant\frac{\log(10Q)}{\log 2}}
			\sum_{\nu_1\leqslant\nu\leqslant\nu_2}\bigg(\frac{QN^3}{qR}\bigg)^{1/2}
			\Bigg(\frac{(qR)^{1/2}\log X}{N^{1/2}}+\frac{M(QN(\log X)^2)^{1/2}}{(qR)^{1/2}}\Bigg)
			\nonumber \\
			\ll & \,\, \sum_{1\leqslant q\leqslant Q}\sum_{0\leqslant j\leqslant\frac{\log(10Q)}{\log 2}}\frac{qR}{N^2}
			\bigg(\frac{QN^3}{qR}\bigg)^{1/2}
			\Bigg(\frac{(qR)^{1/2}\log X}{N^{1/2}}+\frac{M(QN(\log X)^2)^{1/2}}{(qR)^{1/2}}\Bigg)
			\nonumber \\
			\ll & \,\, Q^{5/2}RN^{-1}(\log X)^2+MQ^2(\log X)^2\ll X(\log X)^2,
		\end{align*}
		provided that $M\ll\min\big(X^{117/71}R^{-1}, X^{51/71}\big)$. This completes the proof of Lemma \ref{Type-I-upper}.
	\end{proof}

	\begin{lemma}\label{Type-II-upper}
		Suppose that $a(m)$  and $b(n)$ are sequences supported on the intervals $(M,2M]$ and $(N,2N]$, respectively. Suppose further that
		\begin{equation*}
			\sum_{m\sim M}\big|a(m)\big|^2\ll M(\log M)^{2A},\quad
			\sum_{n\sim N}\big|b(n)\big|^2\ll N(\log N)^{2B},\quad
			A>0,\quad B>0.
		\end{equation*}
		Then, for $X^{10/71}\ll N\ll\min\big(X^{112/71}R^{-1},RX^{-20/71}\big),\,MN\asymp X$, one has
		\begin{equation*}
			S_{II}=\sum_{m\sim M}a(m)\sum_{n\sim N}b(n)e\big(x(mn)^c+y(mn)^d\big)\ll X^{66/71}(\log X)^{A+B+1}.
		\end{equation*}
	\end{lemma}
	\begin{proof}
		Take $Q=[X^{10/71}(\log X)^{-1}]=o(N)$. It follows from Cauchy's inequality and Lemma \ref{I-K-8.17} that
		\begin{equation}\label{S_II-square-upper}
			\big|S_{II}\big|^2\ll\frac{X^2(\log X)^{2(A+B)}}{Q}+\frac{X(\log X)^{2A}}{Q}\sum_{1\leqslant q\leqslant Q}
			\sum_{\substack{n\sim N\\ n+q\sim N}}\big|b(n+q)b(n)\big|\Bigg|\sum_{m\sim M}e\big(f(m,n)\big)\Bigg|,
		\end{equation}
		where
		\begin{equation*}
			f(m,n)=xm^c\Delta(n,q;c)+ym^d\Delta(n,q;d),
		\end{equation*}
		and
		\begin{equation*}
			\Delta(n,q;t)=(n+q)^t-n^t.
		\end{equation*}
		By (\ref{S-upper-2}), we obtain
		\begin{equation}\label{inner-upper}
			\sum_{m\sim M}e\big(f(m,n)\big)\ll q^{1/2}R^{1/2}N^{-1/2}+Mq^{-1/3}R^{-1/3}N^{1/3}.
		\end{equation}
		By noting the fact that $Q=o(N)$, it is easy to see that, for fixed $q$ subject to $1\leqslant q\leqslant Q$, there holds
		\begin{equation}\label{inner-b-upper}
			\sum_{\substack{n\sim N\\ n+q\sim N}}\big|b(n+q)b(n)\big|\ll\sum_{n\sim N}\big|b(n)\big|^2+
			\sum_{\substack{n\sim N\\ n+q\sim N}}\big|b(n+q)\big|^2\ll N(\log N)^{2B}.
		\end{equation}
		Combining (\ref{S_II-square-upper}), (\ref{inner-upper}) and (\ref{inner-b-upper}), we get
		\begin{align*}
			& \,\, \sum_{1\leqslant q\leqslant Q}\sum_{\substack{n\sim N\\ n+q\sim N}}\big|b(n+q)b(n)\big|
			\Bigg|\sum_{m\sim M}e\big(f(m,n)\big)\Bigg|
			\nonumber \\
			\ll & \,\, N(\log N)^{2B}\sum_{1\leqslant q\leqslant Q}
			\big(q^{1/2}R^{1/2}N^{-1/2}+Mq^{-1/3}R^{-1/3}N^{1/3}\big)
			\nonumber \\
			\ll & \,\, N(\log N)^{2B}\big(Q^{3/2}R^{1/2}N^{-1/2}+MN^{1/3}Q^{2/3}R^{-1/3}\big)
			\nonumber \\
			\ll & \,\, (\log N)^{2B}\big(N^{1/2}X^{15/71}R^{1/2}+N^{1/3}X^{233/213}R^{-1/3}\big)\ll X,
		\end{align*}
		provided that
		\begin{equation}
			N\ll\min\big(X^{112/71}R^{-1},RX^{-20/71}\big).
		\end{equation}
		This completes the proof of Lemma \ref{Type-II-upper}.
	\end{proof}

	\begin{lemma}\label{Heath-Brown-identity}
		Let $z\geqslant1$ and $k\geqslant1$. Then, for any $n\leqslant2z^k$, there holds
		\begin{equation*}
			\Lambda(n)=\sum_{j=1}^k(-1)^{j-1}\binom{k}{j}\mathop{\sum\cdots\sum}_{\substack{n_1n_2\cdots n_{2j}=n\\
					n_{j+1},\dots,n_{2j}\leqslant z }}(\log n_1)\mu(n_{j+1})\cdots\mu(n_{2j}).
		\end{equation*}
	\end{lemma}
	\begin{proof}
		See the arguments on pp. 1366--1367 of Heath--Brown \cite{Heath-Brown-1982}.
	\end{proof}

	\section{Exponential sum estimate on}\label{expo-esti-sec}
	In this section, we shall demonstrate the details of the proof of Lemma \ref{S(x,y)-Omega_2}. Set
	\begin{align*}
		\mathfrak{A}=\min\big(X^{117/71}R^{-1},X^{51/71}\big),\qquad \mathfrak{B}=X^{10/71},\qquad
		\mathfrak{C}=\min\big(X^{112/71}R^{-1},RX^{-20/71}\big).
	\end{align*}
	It follows from the inequality $X^{3/4-\eta}\ll R\ll X^{79/71}(\log X)^{-400}$ that
	\begin{align*}
		X^{1/9}\ll X/\mathfrak{A}\leqslant\mathfrak{C},\qquad \mathfrak{B}^2\leqslant\mathfrak{C}.
	\end{align*}
	Trivially, we have
	\begin{align}\label{S(x,y)-trans}
		S(x,y)=\sum_{\lambda X<n\leqslant X}\Lambda(n)e(n^cx+n^dy)+O(X^{1/2}).
	\end{align}
	According to Heath--Brown's identity (i.e., Lemma \ref{Heath-Brown-identity}) with $k=10$, it is easy to see that
	\begin{align*}
		\sum_{\lambda X<n\leqslant X}\Lambda(n)e(n^c x+n^dy)
	\end{align*}
	can be written as linear combination of $O\big((\log X)^{20}\big)$ sums, each of which is of the form
	\begin{align*}
		\mathscr{U}(N_1,\dots,N_{20})
		:= & \,\, \sum_{n_1\sim N_1}\cdots\sum_{n_{20}\sim N_{20}}(\log n_1)\mu(n_{11})\mu(n_{12})\cdots\mu(n_{20})
		\nonumber \\
		& \,\, \qquad\qquad\qquad\qquad\times e\big((n_1\dots n_{20})^cx+(n_1\dots n_{20})^dy\big),
	\end{align*}
	where $N_1\dots N_{20}\asymp X, N_j\leqslant(2X)^{1/10},(j=11,12,\dots,20)$, and some $n_i$ may only take value $1$. Therefore, it is sufficient for us to give upper bound estimates as follows
	\begin{equation*}
		\mathscr{U}(N_1,\dots,N_{20})\ll X^{66/71}(\log X)^{185}.
	\end{equation*}
	Next, we shall consider three cases.
	
	\noindent
	\textbf{Case 1.} If there exists an $N_j$ which satisfies $N_j\gg X/\mathfrak{A}\gg X^{1/9}$, then there must hold $j\leqslant10$ for the fact that $N_j\ll X^{1/10}$ with $j=11,12,\dots,20$. Let
	\begin{equation*}
		m=\prod_{\substack{1\leqslant i\leqslant20\\ i\not=j}}n_i,\qquad n=n_j,\qquad
		M=\prod_{\substack{1\leqslant i\leqslant20\\ i\not=j}}N_i,\qquad N=N_j.
	\end{equation*}
	In this case, we can see that $\mathscr{U}(N_1,\dots,N_{20})$ is a sum of ``Type I'', i.e., $S_I$, subject to
	$M\ll\mathfrak{A}$. By Lemma \ref{Type-I-upper} and a divisor argument, we derive that
	\begin{equation*}
		\mathscr{U}(N_1,\dots,N_{20})\ll X^{66/71}(\log X)^{185}.
	\end{equation*}

	\noindent
	\textbf{Case 2.} If there exists an $N_j$ such that $\mathfrak{B}\leqslant N_j<X/\mathfrak{A}\leqslant\mathfrak{C}$, then we take
	\begin{equation*}
		m=\prod_{\substack{1\leqslant i\leqslant20\\ i\not=j}}n_i,\qquad n=n_j,\qquad
		M=\prod_{\substack{1\leqslant i\leqslant20\\ i\not=j}}N_i,\qquad N=N_j.
	\end{equation*}
	Thus, $\mathscr{U}(N_1,\dots,N_{20})$ is a sum of ``Type II'', i.e., $S_{II}$, subject to
	$\mathfrak{B}\ll N\ll\mathfrak{C}$.
	By Lemma \ref{Type-II-upper}, we derive that
	\begin{equation*}
		\mathscr{U}(N_1,\dots,N_{20})\ll X^{66/71}(\log X)^{185}.
	\end{equation*}

	\noindent
	\textbf{Case 3.} If $N_i<\mathfrak{B}\,(i=1,2,\dots,20)$, without loss of generality, we postulate that
	$N_1\geqslant N_2\geqslant \dots\geqslant N_{20}$. Denote by $\ell$ the natural number such that
	\begin{equation*}
		N_1\dots N_{\ell-1}<\mathfrak{B},\qquad N_1\dots N_{\ell-1}N_\ell\geqslant\mathfrak{B}.
	\end{equation*}
	Since $N_1<\mathfrak{B}$ and $N_{20}<\mathfrak{B}$, then $2\leqslant\ell\leqslant19$. Therefore, there holds
	\begin{equation*}
		\mathfrak{B}\leqslant N_1\dots N_\ell=N_1\dots N_{\ell-1}\cdot N_\ell<\mathfrak{B}
		\cdot\mathfrak{B}\leqslant\mathfrak{C}.
	\end{equation*}
	Let
	\begin{equation*}
		m=\prod_{i=\ell+1}^{20}n_i,\qquad n=\prod_{i=1}^\ell n_i,\qquad
		M=\prod_{i=\ell+1}^{20}N_i,\qquad N=\prod_{i=1}^\ell N_i.
	\end{equation*}
	At this time, $\mathscr{U}(N_1,\dots,N_{20})$ is a sum of ``Type II'', i.e., $S_{II}$, subject to
	$\mathfrak{B}\ll N\ll\mathfrak{C}$.
	By Lemma \ref{Type-II-upper}, we derive that
	\begin{equation*}
		\mathscr{U}(N_1,\dots,N_{20})\ll X^{66/71}(\log X)^{185}.
	\end{equation*}
	Based on the above three cases, we obtain
	\begin{align*}
		\sum_{\lambda X<n\leqslant X}\Lambda(n)e(n^cx+n^dy)
		\ll & \,\, X^{66/71}(\log X)^{185}\cdot(\log X)^{20}
		\nonumber \\
		\ll & \,\, X^{66/71}(\log X)^{205},
	\end{align*}
	which combined with (\ref{S(x,y)-trans}) implies
	\begin{equation*}
		S(x,y)\ll X^{66/71}(\log X)^{205}.
	\end{equation*}
	This completes the proof of Lemma \ref{S(x,y)-Omega_2}.

\section*{Acknowledgement}
The authors would like to appreciate the referee for his/her patience in refereeing this paper.
This work is supported by Beijing Natural Science Foundation (Grant No. 1242003), and
the National Natural Science Foundation of China (Grant Nos. 12471009, 12301006, 11901566, 12001047).

\end{document}